\documentclass{siamltex}
\usepackage{amsmath,amsfonts}
\usepackage{color}
\usepackage{graphicx}
\usepackage{epsfig}

\newcommand{\A}{\mathcal{A}}

\newcommand{\G}{{\mathcal G}}
\newcommand{\iprod}[1]{\langle#1\rangle}

\newcommand{\W}{{\mathcal W}}
\newcommand{\T}{{\mathcal T}}

\newcommand{\I}{{\mathcal I}}

\newcounter{bean}
\newenvironment{romenum}{\begin{list}{{(\roman{bean})}}
{\usecounter{bean}}}{\end{list}}

\renewcommand{\arraystretch}{.75}

\title{A discontinuous Petrov-Galerkin method for time-fractional diffusion equations \thanks{Support of the King Fahd University of Petroleum and Minerals
(KFUPM) through
 the project FT131001 is gratefully acknowledged.}}
\author{K. Mustapha$^1$ \and B. Abdallah$^2$ \and K.M. Furati$^3$ \thanks{
Department of Mathematics and Statistics, KFUPM, Dhahran 31261, Saudi Arabia \texttt{Email$^1$:kassem@kfupm.edu.sa}}}

\date{\today}

\begin{document}

\maketitle

\begin{abstract}
 We propose and analyze a  time-stepping
discontinuous Petrov-Galerkin  method combined with the continuous  conforming finite element method in space for the numerical solution of
time-fractional subdiffusion problems. We prove the existence, uniqueness and stability of approximate solutions, and derive error estimates. To
achieve high order convergence rates from the time discretizations, the time mesh is graded appropriately near~$t=0$  to compensate the singular
(temporal) behaviour of the exact solution near $t=0$ caused by the weakly singular kernel, but the spatial mesh is quasiuniform. In the
$L_\infty((0,T);L_2(\Omega))$-norm ($(0,T)$ is the time domain and $\Omega$ is the spatial domain), for sufficiently graded time meshes,  a
global convergence of order $k^{m+\alpha/2}+h^{r+1}$ is shown, where $0<\alpha<1$ is the fractional exponent, $k$ is the maximum time step, $h$
is the maximum diameter of the spatial finite elements, and $m$ and $r$ are the degrees of approximate solutions in time and spatial variables,
respectively. Numerical experiments indicate that our theoretical   error bound is pessimistic. We observe that the error is of
  order ~$k^{m+1}+h^{r+1}$, that is, optimal in both variables.
\end{abstract}


\pagestyle{myheadings} \thispagestyle{plain} \markboth{K. MUSTAPHA, B. ABDALLAH AND K. M. FURATI}{DPG METHODS FOR FRACTIONAL DIFFUSION
EQUATIONS}
%\allowdisplaybreaks
%%%%%%%%%%%%%%%%%%%%%%%%%%%%%%%%%%%%%%%%%%%5
\setcounter{equation}{0} \setcounter{theorem}{0}\section{Introduction}
%%%%%%%%%%%%%%%%%%%%%%%%%%%%%%%%%%%%%%%%%%%%
In this paper, we propose and analyze the  time-stepping discontinuous Petrov-Galerkin (DPG)  method combined with the standard continuous
finite element  (DPG-FE) method in space for solving numerically the time-fractional diffusion model:
\begin{equation}
\label{eq:fractional equation}
 ^c{\rm D}^{1-\alpha} u(x,t)-\Delta u(x,t) = f(x,t) \quad\mbox{ for } (x,t)\in \Omega\times (0,T]~~~{\rm with}~~u(x,0)=u_0(x),
\end{equation}
subject to homogeneous Dirichlet spatial boundary conditions. Here,   $\Omega \subset \mathbb{R}^d$ (with $d=1,2,3$) is a convex polyhedral domain with boundary $\partial \Omega$, $f$ and $u_0$ are given
functions assumed to be sufficiently regular such that the solution $u$ of \eqref{eq:fractional equation} is in the space $W^{1,1}((0,T);H^2(\Omega))$; see  the regularity analysis in \cite{McLean2010} (further regularity assumptions will be imposed later),  and $T>0$ is a fixed value. Here, $^c{\rm D}^{1-\alpha}$ denotes the time fractional Caputo derivative of order $\alpha$  of the
function $u$ defined by
\begin{equation}  ^c{\rm D}^{1-\alpha} u(x,t):= (\I^{\alpha} u') (x,t)\quad {\rm with}~~
0<\alpha<1,
\end{equation}
where $u':=\frac{\partial}{\partial t} u$ and $\I^\alpha$ is the Riemann--Liouville time fractional integral operator;
\begin{equation}  \I^\alpha v(t):=  \int_0^t\omega_{\alpha}(t-s)v(s)\,ds\quad\text{with}
\quad \omega_{\alpha}(t):=\frac{t^{\alpha-1}}{\Gamma(\alpha)}\,.
\end{equation}
%with   $\Gamma$ being  the
% gamma function.

Problems of the form \eqref{eq:fractional equation} arise in a variety of physical, biological and chemical applications
\cite{Mathai,podlubny,SmithMorrison,Tarasov}. Problem \eqref{eq:fractional equation} describes slow or anomalous sub-diffusion and occurs, for
example, in models of fractured or porous media, where the particle flux depends on the entire history of the density gradient$\nabla u$.

\subsection{Motivation and outline of the paper} The nonlocal nature of the fractional derivative operator $^c{\rm D}^{1-\alpha}$  means that on each time subinterval, one must efﬁciently evaluate a sum of integrals over all previous time subintervals.
%Note that direct implementation of the time-stepping DPG method requires $\mathcal{O}(mN^2 {\bf M})$ ($N$ is the number of time-mesh elements  and $m$ is the
%degree of the DPG solution and ${\bf M}$ is the degree of freedom from the spatial discretizations) operations and requires $\mathcal{O}(m N {\bf M})$ storage in the presence of a memory term.
Thus, reducing the number of time steps as much as possible and  maintaining high accuracy of the discrete solutions are important. So, the most obvious thing is to propose efficient high order methods for the model problem \eqref{eq:fractional equation}. In this work we investigate (for the first time to the best of our knowledge) a high order accurate (unconditionally stable) time-stepping numerical method for problem \eqref{eq:fractional equation}. However, due to the typical singular behaviour of ~$u$ near ~$t\to0$ ~\cite{McLean2010,McLeanMustapha2007}, high order methods can fail to achieve fast convergence. To this end, we propose to deal with the accuracy issue (in time) by developing   a high order DPG method that allows for the singular behaviour of~$u$, by employing non-uniform time steps.  An important feature of the DPG method is that it allows for locally varying time-steps and
   approximation orders which are beneficial to handle problems with low regularity.  The DPG method was introduced initially for solving first order ODEs in \cite{Hulme}. Later on,  DPG methods were investigated by several authors for solving various problems. For instance, for  advection-diffusion and elliptic problems, see \cite{BottassoMichelettiSacco2002, BottassoMichelettiSacco2005}, for transport equations see \cite{DemkowiczGopalakrishnan2010}, and refer to
\cite{LinLinRaoZhang, Mustapha2008} for Volterra integro-differential equations with smooth memory. Here, we extend the original DPG method in
\cite{Hulme} to discretize in time the fractional diffusion problem \eqref{eq:fractional equation}. For the sake of completeness, we combine the time-stepping DPG with the continuous finite elements (FEs) in space, which will then define a fully discrete computable scheme.   Existence, uniqueness and stability of our numerical scheme will be provided. For the error analysis, we show convergence rates of order $O(k^{m+\alpha/2} + h^{r+1})$ in the $L_\infty((0,T);L_2(\Omega))$-norm, where $k$ is the maximum time step, $h$ is the maximum diameter of the spatial finite elements, and $m$ and $r$ are the degrees of approximate solutions in time and spatial variables, respectively. The main difficulty in our stability and error analysis is due to the
trouble from the time  discretization. In this direction, we make  full use of several important properties of the operator  $^c{\rm D}^{1-\alpha}$; see Lemma
\ref{lemma:positivity-equal-0}.  In contrast, for $m=1$,  the considered time stepping DPG scheme amounts to a generalized post-processed Crank-Nicolson scheme. To validate the achieved theoretical results,  a series of numerical results will be given at the end of the
paper. Since  in the present work our emphasis is on convergence properties rather than algorithmic implementation, in
our numerical experiments we do not use the fast algorithm.  A direct implementation of the considered  method requires $\mathcal{O}(mN^2 {\bf M})$  operations and requires $\mathcal{O}(m N {\bf M})$ storage, owing  to the presence of the  memory term, where $N$ is the number of time mesh elements and ${\bf M}$ is the spatial degrees of freedom. Proposing a fast algorithm for evaluating the discrete solution  is beyond the scope of the present paper. This will be a topic for future research.

The outline of the paper is as follows.  Section~\ref{sec:numerical-method} introduces a fully discrete DPG-FE scheme. In
Section~\ref{sec:existence-uniqueness}, using appropriately  the positivity, coercivity, and  continuity properties of the operator $^c{\rm
D}^{1-\alpha}$, we prove the existence, uniqueness, and stability of the discrete solution. The error and convergence analysis are given in
Section~\ref{sec:error}.   We derive error estimates, which are completely explicit in the local step sizes, the local polynomial degrees, and
the local regularity of the analytical solution. Using suitable refined time-steps (towards $t=0$), in the $L_\infty((0,T);L_2(\Omega))$-norm,
convergence of order $O(k^{m+\alpha/2}+h^{r+1})$ will be achieved. Section~\ref{sec:implementation and numerics} is devoted to present a series
of numerical tests which indicate the validity of our theoretical convergence properties  and also illustrate that our error bounds are
pessimistic. For a strongly graded time mesh, we observe that the error from the time discretization is $O(k^{m+1})$ (optimal), which is better
than our theoretical estimate by a factor~$k^{1-\alpha/2}$.
%Several authors have considered numerical methods for \eqref{eq:fractional equation}, see
%\cite{Cui2012,GaoSun2011,JinLazarovZhou2012,Quintana-MurilloYuste2011,WangaWanga2011,ZhangSun2011} and references therein.
\subsection{Literature review}
Several authors have proposed a variety of low-order numerical methods for the model problem \eqref{eq:fractional equation}.  For one
dimensional cases, \cite{ZhaoSun2011} constructed a box-type scheme based on combining order reduction approach and L$_1$ discretization was
considered.  The authors proved global convergence rates of order $O(k^{2-\alpha} + h^2)$, assuming that the solution $u$ of
\eqref{eq:fractional equation} is sufficiently regular. An implicit finite difference scheme in time and Legendre spectral methods in space were
studied in \cite{LinXu2007}. Stability and convergence of order $O(k^{1+\alpha}+r^{-\ell})$ of the method were established, where $r$ is the
degree of the approximate solution in space and $\ell$ is related to the order of regularity of the solution $u$ of \eqref{eq:fractional
equation}, which is typically low. An extension of this work was considered in \cite{LiXu2009} where  a time-space spectral method has been
proposed and analyzed. For an explicit difference (first order in time-second order in space) method, we refer the reader to
\cite{Quintana-MurilloYuste2011}. The stability analysis was carried out by means of a kind of fractional von Neumann method. The authors
provided a partial convergence analysis (truncation error of order $O(k + h^2)$) assuming that $u$ is sufficiently regular. An implicit
Crank--Nicolson had been considered in \cite{SweilamKhaderMahdy2012} and the stability of the proposed scheme was shown. Some numerical
experiments were presented to illustrate the convergence of the approximate solutions. Very recently, two finite difference/element approaches
were developed in \cite{ZengoLiLiuTurner2013}, in which the time direction was approximated by the fractional linear multistep method and the
space direction was approximated by the standard FEM of degree $r$. Assuming the  solution  of \eqref{eq:fractional equation} is sufficiently
smooth, convergence rates of order $O(k^{1+\alpha}+h^{r+1})$  were proved.

For two (or three)  dimensional cases,  a standard central difference approximation was used for the spatial discretization, and, for the time stepping, two
alternating direction implicit (ADI) schemes based on the $L1$ approximation and backward Euler method  were investigated in
\cite{ZhangSun2011}. Assuming that $u$ is smooth, the authors proved  convergence of order $O(k^{\min\{2\alpha,2- \alpha\}} + h^2)$ and
$O(k^{\min\{1+\alpha,2- \alpha\}} + h^2)$, respectively. A compact finite difference method with operator-splitting techniques was considered in
\cite{Cui2012}.
 The Caputo derivative was evaluated by the $L1$ approximation, and the second order spatial derivatives
  were approximated by the fourth order compact (implicit) finite differences.
The unconditional stability was analyzed, and by using the energy method,  errors of order $O(k^{\min\{1+\alpha,2- \alpha\}} +h^4)$ were
achieved assuming that $u$ is smooth. In \cite{JinLazarovZhou2012}, for $f=0$ in problem \eqref{eq:fractional equation} (that is, homogeneous
case), the authors studied two spatial semidiscrete  piecewise linear approximation schemes: Galerkin FEM and lumped mass Galerkin FEM. Optimal
error estimates   were established including the case of non-smooth initial data.
In \cite{JinLazarovZhou2014}, the same authors  developed two simple fully discrete schemes based on  Galerkin FEMs in space and implicit backward differences for the  time discretizations. Optimal error estimates with respect to the regularity of the initial data were established.

In contrast, for the numerical solutions of the alternative representation of  the fractional subdiffusion problem \eqref{eq:fractional
equation}:
\begin{equation}
\label{eq: reimann} u'(x,t) - {^{R}{\rm D}}^{1-\alpha} \Delta u(x,t)  = f(x,t) \quad\mbox{ for } (x,t)\in \Omega\times (0,T],
\end{equation}
where  $^{R}{\rm D}^{1-\alpha} u(x,t):=\frac{\partial }{\partial t}( \I^\alpha u)(x,t)$ (Riemann--Liouville fractional time derivative of $u$),
we refer the reader to
\cite{ChenLiuAnhTurner2011,CuestaLubichPalencia2006,Cui2009,LanglandsHenry2005,LiuYangBurrage2009,McLeanMustapha2009,Mustapha2011,MustaphaAlMutawa,MustaphaMcLean2012,MustaphaMcLean2013,YusteAcedo2005,ZhuangLiuAnhTurner2008}.
Practically, the two representations are different ways of writing the same equation as they are equivalent under reasonable assumptions on the
initial data, see \cite{YusteQuintana2009}. However, the numerical methods obtained for each representation are formally different.
%%%%%%%%%%%%%%%%%%%%%%%%%%%%%%%%%%%%%%%%%%%%%%%%%%%%%%%%%%
\setcounter{equation}{0} \setcounter{theorem}{0}
\section{Numerical scheme}\label{sec:numerical-method}
%%%%%%%%%%%%%%%%%%%%%%%%%%%%%%%%%%%%%%%%%%%%%%%%%%%%%%%%%
To describe our fully discrete DPG-FE method, we introduce a (possibly nonuniform) time partition of the interval $[0,T]$ given by the points:
$0=t_{0}<t_{1}<\cdots<t_{N}=T\,.$ We set $I_n=(t_{n-1},t_n)$ and $k_n=t_n-t_{n-1}$ for $1\le n\le N$. Let $S_h\subseteq H_0^1(\Omega)$
 :=$\{v\in H^1(\Omega):~~~v=0~{\rm on}~\partial \Omega\}$ denotes the space of continuous, piecewise polynomials
 of degree $\le r$ ($r\ge 1$) with respect to a quasi-uniform partition of~$\Omega$ into
conforming triangular  finite elements, with maximum diameter~$h$. Hence, the Ritz projection operator $\,R_h:H^1_0(\Omega) \to S_h$ defined by
 \begin{equation}\label{eq: Ritz projection} \iprod{\nabla (R_h v-v),\nabla \chi}=0\quad {\rm
for~ all}~~~ \chi\in S_h,\end{equation}  has the approximation property: for $v \in H^{s+1}(\Omega) \cap  H^1_0(\Omega)$,
\begin{equation}\label{eq:estimate ritz projector} \|  R_hv-v\|\leq C\,%\frac{h^{\min\{s,r\}+1}}{r^{s+1}}
h^{\min\{s,r\}+1}\|v\|^2_{s+1}\quad {\rm for}~~r\geq 1 ~{\rm and}~~s\geq 0,
\end{equation}
where, by  $\iprod{\cdot,\cdot}$ and $\|\cdot\|$,  we denote the $L_2$-inner product and the associated norm over the spatial domain $\Omega$.
By $\| \cdot\|_r$ we denote the standard $H^r(\Omega)$-norm for $ r\ge 1$.
%For later use,   the norm $\|
%\cdot\|_{L_{\infty}((0,T);L_2(\Omega))}$ is denoted by
%$\| \cdot\|_{L_{\infty}(L_2)}$. %By $C$ we denote a generic constant that is independent

Next,  we introduce the following spaces: for a fixed $m\ge 1,$
\begin{equation}
\begin{split}
 \W(S_h)&=\{{v\in \mathcal{C}([0,T]; S_h):~~v|_{I_{n}}\in P_{m}(S_h)~{\rm for}~1\le n\le N}\}\label{eq: trial space}\\
 \T({S_h})&=\{v \in L_2((0,T); S_h):~~v|_{I_{n}}\in
P_{m-1}(S_h)~{\rm for}~1\le n\le N\} \end{split}\end{equation}  where $P_{m}(S_h)$ denotes the space
of polynomials of degree $\le m$ in the time variable $t$,  with coefficients in~$S_h$. So, for a given function $v\in \W(S_h)$ then  $v'\in \T(S_h)$. Here  $v'$ is a piecewise polynomial obtained by differentiating $v$ with respect to $t$ on each subinterval $I_n$ for $1\le n\le N$.

Now, we are ready to define our  DPG-FE numerical scheme for problem \eqref{eq:fractional equation} as follows:
 Find   $U_h \in \W(S_h)$ such that,  $U_h(0)=R_h u_0$, and
 \begin{equation}
 \label{eq: DPGFM}
\int_0^T\left(\iprod{^c{\rm D}^{1-\alpha} U_h ,X}+\iprod{ \nabla U_h,  \nabla X}\right)dt=\int_0^T \iprod{f,X}\,dt\quad \forall~\text{ $X \in
\T(S_h)$ }.
 \end{equation}
 In the next section,  we will  show the well-posedness of our scheme.
% where the global bilinear form
%\begin{equation*}G(U_h,X):=\int_0^T\iprod{^c{\rm D}^{1-\alpha} U_h ,X}\,dt+\int_0^T\iprod{ \nabla U_h,  \nabla X}\,dt\,. \end{equation*}
%Noting that,  since the solution $u$ of problem \eqref{eq:fractional equation} satisfies
%\begin{equation*} G(u,X)=\int_0^T\iprod{f,X}\,dt\quad \text{ for all $X \in \T(S_h)$}\,,
% \end{equation*}
% $U_h-u$  satisfies the orthogonality condition
% \begin{equation}\label{eq:orthogonality property}
% G(U_h-u ,X)=0\quad \text{ for all $X \in \T(S_h)$}\,.
% \end{equation}
 %%%%%%%%%%%%%%%%%%%%%%%%%%%%%%%%%%%%%%%%%%%%%%%%%%%%%%%%%%%%%%%%{sec:existence-uniqueness}
 \setcounter{equation}{0}
 \setcounter{theorem}{0}
\section{Well-posedness  of the DPG-FE  scheme}\label{sec:existence-uniqueness}
%%%%%%%%%%%%%%%%%%%%%%%%%%%%%%%%%%%%%%%%%%%%%%%%%%%%%%%%%%
In this section, we show the well-posedness of the discrete DPG-FE solutions.  To be able to do this, we need to carefully use several  crucial
properties of the operator ${^c}{\rm D}^{1-\alpha}$. These properties will be stated in the next lemma, we refer the reader to \cite[Lemma
3.1]{MustaphaSchoetzau2012} for the proof.
\begin{lemma}
\label{lemma:positivity-equal-0} For $1\le j\le n,$ let $v|_{I_j} ,\,w|_{I_j}   \in   H^1(I_j, L_2(\Omega))\cap \mathcal{C}(\overline I_j, L_2(\Omega))$. {There holds:}
\begin{romenum}
\item If $\max_{j=0}^n \,\|v(t_j)\|  +\int_0^{t_n}\iprod{ v', ^c{\rm D}^{1-\alpha} v}\,dt =0,$
then $v\equiv 0$ on $[0,t_n]$.
\item The coercivity property:
$$
\int_0^{t_n} \iprod{v',{^c}{\rm D}^{1-\alpha} v}\,dt \ge c_\alpha  \int_0^{t_n} \|^c{\rm D}^{1-\frac{\alpha}{2}} v\|^2\,dt~{\rm with}
~c_\alpha=\cos(\alpha \pi/2).
$$
\item The continuity property: for any  $\epsilon >0,$
$$
\Big|\int_0^{t_n}\iprod{v',{^c}{\rm D}^{1-\alpha} w}\,dt\Big| \le \int_0^{t_n} \left(\frac{\epsilon}{2\,c_\alpha^2} \iprod{v',{^c}{\rm
D}^{1-\alpha} v}+\frac{1}{2\,\epsilon }\iprod{w',{^c}{\rm D}^{1-\alpha} w}\right)dt$$
where assuming that $\int_0^{t_n} \iprod{v',{^c}{\rm D}^{1-\alpha} v}\,dt$ and $\int_0^{t_n} \iprod{w',{^c}{\rm D}^{1-\alpha} w}dt$ to be absolutely bounded should be sufficient for this property.
\end{romenum}
  \end{lemma}

 Next, we  prove the existence and uniqueness of the
DPG-FE solution.
  \begin{theorem}\label{theorem: existence}
Assume that $f \in L_2((0,T);L_2(\Omega))$ and $u_0 \in H^1_0(\Omega)$. Then, the discrete solution $U_h$ of \eqref{eq: DPGFM} exists and is
unique.
\end{theorem}
\begin{proof}
Because of the finite dimensionality of problem~\eqref{eq: DPGFM}  on each sub-domain $\Omega\times I_n$, the existence of the approximate
solution  $U_h$ follows from its uniqueness.  To show the uniqueness,  we take $X\equiv 0$ outside  $\Omega \times I_n$  in~\eqref{eq: DPGFM},
then we find that
\begin{equation}\label{eq:DG-step-algebraic}
    \int_{I_n}\,
    \Bigl(\iprod{{^c}{\rm D}^{1-\alpha} U_{h},X} +\iprod{\nabla  U_{h},\nabla X}  \Bigr)\,dt
=\int_{I_n}\iprod{f,X}\,dt~~~{\rm with}~~ U_h(0)=R_h u_0\,.
\end{equation}
 Since
$U_{h}$ is constructed element by element (in time), it is enough to show the uniqueness on the first sub-domain  $\Omega \times I_1$.
%That is,
%it is enough to consider $n=1$ in~\eqref{eq:DG-step-algebraic} (for $n\ge 2$ the proof is completely analogous).
 To this end, let $U_{h,1}$ and
$U_{h,2}$ be two solutions of \eqref{eq:DG-step-algebraic} on $\Omega\times I_1$. By linearity, the difference $V_h:=U_{h,1}-U_{h,2}$ on
$\Omega\times I_1$  satisfies:
\begin{equation}
\label{eq: DG solution for V}
    \int_0^{t_1}\left(\iprod{ ^{c}{\rm D}^{1-\alpha} V_h, X}+\iprod{\nabla  V_h,\nabla X}\right)\,dt
=0~~ {\rm for~all}~X \in P_{m-1}(S_h)
\end{equation}
with $V_h(0)=0\,.$ Choosing $X=V_h'\in P_{m-1}(S_h)$ yields
 \[
\int_0^{t_1}\iprod{ ^{c}{\rm D}^{1-\alpha} V_h, V_h'}\,dt +\frac{1}{2}\int_0^{t_1}
    \frac{d}{d t}\|\nabla V_h(t)\|^2\,dt=0.
\]
Integrating, then  using $V_h(0)=0$ and the positivity $\int_0^{t_1}\iprod{^{c}{\rm D}^{1-\alpha}V_h,V_h'}\,dt\ge 0$; see property $(ii)$ in
Lemma \ref{lemma:positivity-equal-0}, we conclude that $ \|\nabla V_h(t_1)\|^2=\|\nabla V_h(0)\|^2 =0$ and $\int_0^{t_1}\iprod{^{c}{\rm
D}^{1-\alpha}V_h,V_h'}\,dt=0. $ Therefore, $\|V_h(t_1)\|=\|V_h(0)\|=0$ and consequently,   an application of Lemma
\ref{lemma:positivity-equal-0} $(i)$ yields $V_h\equiv 0$ on $\Omega \times [0,t_1]$. This completes the proof.
\end{proof}

In the next theorem, the stability of the DPG-FE scheme will be shown.
\begin{theorem}\label{theorem: stability}
Assume that $f \in H^1((0,T);L_2(\Omega))$ and $u_0 \in H^1_0(\Omega)$ in problem \eqref{eq:fractional equation}. Then, for $1\leq n \leq N$, the DPG-FE solution $U_h$ of~\eqref{eq: DPGFM}
satisfies:
\begin{equation*} \int_0^{t_n}\iprod{^{c}{\rm D}^{1-\alpha} U_h,U_h'}\,dt+\| \nabla U_h(t_n)\|^{2}
 \le \|\nabla u_0\|^2+\frac{1}{c_\alpha^2}\,  \int_0^{t_n}\iprod{f, {^R}{\rm D}^{\alpha} f}\,dt,
\end{equation*}
where  $c_\alpha$ is the constant in Lemma~\ref{lemma:positivity-equal-0}\,.
\end{theorem}
\begin{proof}
For $1\le n\le N$, we choose $ X|_{[0,t_n]}=U_h'$ and zero elsewhere  in \eqref{eq: DPGFM}, and use the identity $f={^{c}{\rm
D}}^{1-\alpha}(\I^{1-\alpha} f)$ (since $ \;f \in H^1((0,T);L_2(\Omega))$\,),
 \begin{equation}\label{eq: local scheme}
\int_0^{t_n}\left(\iprod{^{c}{\rm D}^{1-\alpha} U_h,U_h'}+ \iprod{ \nabla U_h , \nabla U_h'}\right)dt= \int_0^{t_n}\iprod{ ^{c}{\rm
D}^{1-\alpha}(\I^{1-\alpha} f),U_h'}\,dt,
\end{equation}
  But, from the continuity property $(iii)$ of Lemma
\ref{lemma:positivity-equal-0},
\begin{align*} 2\Big|\int_0^{t_n} \iprod{^{c}{\rm D}^{1-\alpha}(\I^{1-\alpha} f) , & U_h'}\,dt\Big|
\\&\le \int_0^{t_n}\left( \frac{1}{c_\alpha^2}\iprod{^{c}{\rm D}^{1-\alpha}(\I^{1-\alpha} f), (\I^{1-\alpha} f)'}+\iprod{^{c}{\rm
D}^{1-\alpha} U_h,U_h'}\right)dt\\
&= \frac{1}{c_\alpha^2}\, \int_0^{t_n}\iprod{f, {^R}{\rm D}^{\alpha} f}\,dt+ \int_0^{t_n}\iprod{^{c}{\rm D}^{1-\alpha} U_h,U_h'}\,dt\,.
\end{align*}
Inserting this in \eqref{eq: local scheme} and  using; $2\int_0^{t_n}\iprod{\nabla U_h , \nabla U_h'}\,dt =\| \nabla U_h(t_n)\|^2-\| \nabla
U_h(0)\|^{2},$ will complete the proof.
\end{proof}
%%%%%%%%%%%%%%%%%%%%%%%%%%%%%%%%%%%%%%%%%%%%%%%%%%%%%%
 \setcounter{equation}{0}
 \setcounter{theorem}{0}
\section{Error Analysis}\label{sec:error}
%%%%%%%%%%%%%%%%%%%%%%%%%%%%%%%%%%%%%%%%%%%%%%%%%%%%%
In this section, we carry out  a priori error analysis of the DPG-FE method \eqref{eq: DPGFM}.  The starting point is to  introduce a projection
operator that has been used various times in the  analysis of several numerical methods.
%the  discontinuous Galerkin
%time-stepping methods, and FE methods for one dimensional reaction-diffusion problems.
\subsection{Projection and errors}
For $2\le n\le N$ and for $m\ge 1$,  the (Raviart-Thomas) projection operator $\Pi : \mathcal{C}(\overline I_{n};H^\ell(\Omega)) \to
\mathcal{C}(\overline I_{n};P_m(H^\ell(\Omega))$  defined by:
 \[\Pi u(t_{j})=u(t_{j}) ~~{\rm for}~~j=n-1,\,n~~~~{\rm and}~~\int_{I_n} \iprod{\Pi u-u,\,v}\,dt=0~~\forall~~v \in P_{m-2}(L_2(\Omega))\,.\]
Here  $\ell\ge 0$, $P_{m}(H^\ell(\Omega))$ is the space
of polynomials of degree $\le m$ in the time variable $t$,  with coefficients in~$H^\ell(\Omega)$\,. On $I_1$, due to the singular behaviour of $u$ at $t=0$ in the model problem
\eqref{eq:fractional equation},  we let   $\Pi u|_{I_1}$ be  a linear polynomial in the time variable that
interpolates $u$ at the end nodes; $t_0$ and $t_1$.

Notice that, since  $(\Pi u)'|_{I_1}$ is independent of $t$ and since $(\Pi u-u)|_{t=0,t_1}=0$,
\[\int_{I_1} \iprod{(\Pi u-u)'(t),\,(\Pi u)'(t)}\,dt
=0\,.\]
However, for $n\ge 2,$  an integration by parts yields
\[\int_{I_n} \iprod{(\Pi u-u)'(t),\,(\Pi u)'(t)}\,dt
=-\int_{I_n} \iprod{(\Pi u-u)(t),(\Pi u)''(t)}\,dt=0\,.\]
Hence, using these facts and the Cauchy-schwarz inequality, we obtain
\[\int_{I_n} \|(\Pi u)'(t)\|^2\,dt=\int_{I_n} \iprod{u'(t),(\Pi u)'(t)}\,dt
\le \frac{1}{2}\int_{I_n}(\|(\Pi u)'(t)\|^2+\|u'(t)\|^2)\,dt\,.\] Therefore, the projection operator $\Pi$ has the  following property: for $1\le n\le N$,
\begin{equation}\label{(ii)}
\int_{I_n}\|(\Pi u)'(t)\|^{2} dt \leq 2\int_{I_n}\|u'(t)\|^{2}\,dt\quad{\rm for~ any}~~ u\in H^1(I_n;L^2(\Omega))\,.
\end{equation}

In the next theorem, we state the error estimates of the projection operator $\Pi$. For convenience, we
introduce the notations:
\[
\|\phi\|_{I_n}:=\|\phi\|_{L_\infty(I_n,L_2(\Omega))}=\sup_{t\in I_n}\|\phi(t)\|\,.
\]
%Furthermore, by $C$ we denote a generic constant that is independent of $h$ and $k$, but may depends on $m,$ $r,$, $\alpha$, and $T$.
\begin{theorem}\label{thm: estimate of eta} For  $u|_{I_n}\in  H^{m+1}(I_n;L_2(\Omega))$ with $2\leq n\leq N$, we have
\begin{equation*} \|\Pi u-u\|_{I_n}^2+k_n^2\|(\Pi u-u)'\|_{I_n}^2\leq C_{m} k_{n}^{2m+1}\int_{I_n}\|u^{(m+1)}(t)\|^{2}\,dt\, ~~{\rm for}~~m\ge 1\,.
\end{equation*}
\end{theorem}
\begin{proof}
First, for $m=1$, on the subinterval $I_n$, $\Pi u$ is a linear polynomial in time that interpolates $u$ at the end points of $I_n$. Thus,  for $t\in I_n$,
\begin{equation}\label{eq:estimate eta}
 \Pi u(t)
=u(t)+\frac{1}{k_{n}}\int_{I_n}\int_t^{t_n}[u'(s)-u'(q)]\,ds\,dq~~{\rm  and}~~ (\Pi u)'(t) =\frac{1}{k_{n}}\int_{I_n}u'(s)ds\,.
\end{equation}
Using this representation of $\Pi u$,  we can easily derive the desired estimate.

For $m\ge 2,$   we recall first the following error estimate properties of the projection operator $\Pi$ (refer for example to \cite[Chapter 3]{Schwab98}
for the proof):
\begin{align*} \int_{I_n}\|(\Pi u-u)'(t)\|^{2}\,dt &\leq C_{m}
k_n^{2m}\int_{I_n}\|u^{(m+1)}(t)\|^{2}\,dt,
\end{align*}
Then,  by the equality:
$ (\Pi u-u)(t)= \int_{t_{n-1}}^{t}(\Pi u-u)'(s)\,ds$,  the Cauchy-Schwarz inequality, and the above estimate, we have
\begin{multline*} \|\Pi u-u\|_{I_n}^2  \leq \Big(\int_{I_n}\|(\Pi u-u)'(s)\|ds\Big)^2
 \\ \leq k_n \int_{I_n}\|(\Pi u-u)'(s)\|^{2}ds
 \leq C_{m} k_n^{2m+1}\int_{I_n}\|u^{(m+1)}(t)\|^{2}\,dt \,.
\end{multline*}
To estimate  $\|(\Pi u-u)'\|_{I_n}$, we decompose it as: \begin{equation}\label{eq: decompose of pi u} \|(u-\Pi u)'\|_{I_n}\le \|(u-\tilde \Pi
u)'\|_{I_n}+\|(\tilde \Pi u-\Pi u)'\|_{I_n}\end{equation}
 where $\tilde \Pi u|_{I_n} \in P_{m}(L_2(\Omega))$ will be defined such that
\begin{align*} \|u-\tilde \Pi u\|_{I_n}^2+k_n^2\|(u-\tilde \Pi u)'\|_{I_n}^2 &\leq C_{m}
k_n^{2m+1}\int_{I_n}\|u^{(m+1)}(t)\|^{2}\,dt\,.
\end{align*}
For instance, one may choose $\tilde \Pi u|_{I_n}$ as follows: $\tilde \Pi u$ interpolates $u$ at $t_{n-1}$ and $t_n$,
\[(\tilde \Pi u)'(t_n)=u'(t_n) \quad {\rm and}\quad  \tilde \Pi u(\xi_{n,\ell})=u(\xi_{n,\ell}),\quad \ell=1,\cdots,m-2,\]
where $\xi_{n,\ell}=t_{n-1}+k_n \xi_\ell$ and $0<\xi_1<\xi_2<\cdots<\xi_{m-2}<1$ are the $(m-2)$-point Gauss-Legendre quadrature on the interval
$(0,1)$.

  Since
$\|(\tilde \Pi u-\Pi u)'\|_{I_n}\le C_mk_n^{-1}(\|\tilde \Pi u-u\|_{I_n}+\| u-\Pi u\|_{I_n})$ (by the inverse and triangle inequalities), from
\eqref{eq: decompose of pi u}, we have
\begin{align*} \|(\Pi u-u)'\|_{I_n}^2&\le
2\|(u-\tilde \Pi u)'\|_{I_n}^2+C_mk_n^{-2}(\|\tilde \Pi u-u\|_{I_n}^2+\| u-\Pi u\|_{I_n}^2) \\
&\le
   C_m  k_{n}^{2m-1} \int_{I_n}\|u^{(m+1)}(t)\|^{2}\,dt
\end{align*}
and therefore, the proof is completed now. \end{proof}
\subsection{Error decomposition and an interesting bound}
To estimate the error $U_h-u$, we decompose it into three terms (using the operators $R_h$ and $\Pi$) as follows:
\begin{equation}\label{eq: decompose Uh -u}
U_h-u=\zeta+\Pi\xi+\eta:=(U_h-\Pi R_hu)+\Pi(R_h u-u)+(\Pi u-u)\,.
\end{equation}
Since the Ritz projection estimate in \eqref{eq:estimate ritz projector} and the first estimate in Theorem \ref{thm: estimate of eta}  can be
used to bound $\Pi \xi$ and $\eta$, the main task reduces to bound  $\zeta.$ To do so, we derive next an interesting upper bound of $\zeta$ that
depends on $\eta$ and $\xi$ where we assume that $ u  \in W^{1,1}((0,T);H^2(\Omega))$. To satisfy this property we let
$u_0 \in H^1_0(\Omega) \cap H^{2+\epsilon_1}(\Omega)$ and $\int_0^t s^j   \|\frac{\partial^j}{\partial s^j}{^{R}{\rm D}}^{\alpha} f(s)\|_2\,ds\le t^{\epsilon_2}$ for some $\epsilon_1,\epsilon_2 >0$ with $j=0,\,1,\,2.$ One way to see this is to rewrite the model  problem \eqref{eq:fractional
equation} as:
$u'- {^{R}{\rm D}}^{\alpha} \Delta u  = {^{R}{\rm D}}^{\alpha} f$ and then we refer to \cite[Theorems 4.4 and 5.7]{McLean2010}.
\begin{theorem}\label{thm: zeta estimate} Assume that $ u  \in W^{1,1}((0,T);H^2(\Omega))$. Then, for $1\le n\le N,$ we have
\begin{multline}\label{eq: zeta estimate}
\int_0^{t_n}\iprod{^{c}{\rm D}^{1-\alpha} \zeta,\zeta'}\,dt+ \|\nabla \zeta(t_n)\|^{2} \\ \le   \frac{4}{c_\alpha^2}\left(
\int_0^{t_n}\left(\iprod{^{c}{\rm D}^{1-\alpha} \eta,\eta'}+\iprod{\Delta \eta,{^{c}{\rm D}}^{\alpha} \Delta \eta}+ \omega_{\alpha+1}^2(t_n)
\|\xi'\|^2\right)dt\right) \,.
\end{multline}
\end{theorem}
\begin{proof}
The DPG-FE scheme \eqref{eq: DPGFM} and the decomposition in \eqref{eq: decompose Uh -u} imply
\begin{multline*}
 \int_0^T\left(\iprod{^{c}{\rm D}^{1-\alpha} \zeta,X}+\iprod{\nabla \zeta ,\nabla X}\right)dt
\\
=-\int_0^T\left(\iprod{^{c}{\rm D}^{1-\alpha} (\Pi\xi+\eta),X}+\iprod{\nabla(\Pi\xi+\eta) ,\nabla X}\right)dt\,. \end{multline*} But,  $\Pi$
commutes with $R_h$ ($\Pi\,R_h= R_h\,\Pi$) and so,  from the definition of Ritz projector, we have $\iprod{\nabla \Pi \xi,\nabla
X}=\iprod{\nabla (R_h (\Pi u)-\Pi u),\nabla X}=0.$ Hence,
\[
\int_0^T\left(\iprod{^{c}{\rm D}^{1-\alpha} \zeta ,X}+\iprod{\nabla\zeta,\nabla X}\right)dt=-\int_0^T\iprod{^{c}{\rm D}^{1-\alpha}
(\Pi\xi+\eta)-\Delta \eta , X}\,dt\,. \]
 Now, choosing $ X|_{(0,t_n)}=\zeta'$ and zero elsewhere,  then
%  with the aid of the  equality:
%$$2\iprod{\nabla \zeta ,\nabla \zeta'}_{t_n}=\|\nabla \zeta^n\|^{2}-\|\nabla \zeta^0\|^{2}$$ and
using $2\int_0^{t_n} \iprod{\nabla\zeta,\nabla X}\,dt =\|\nabla \zeta(t_n)\|^2- \|\nabla \zeta(0)\|^2 =\|\nabla \zeta(t_n)\|^2$, and the identity
$\Delta \eta(t) ={^{c}{\rm D}}^{1-\alpha}(\I^{1-\alpha} \Delta \eta)(t)$ for $t \in [0,t_n]$  (follows because $\Delta \eta \in W^{1,1}((0,t_n),L_2(\Omega))$\,),
we obtain
\begin{equation}\label{eq:orthogonality}
2\int_0^{t_n} \iprod{^{c}{\rm D}^{1-\alpha} \zeta,\zeta'}\,dt+ \|\nabla \zeta(t_n)\|^{2} =-2\int_0^{t_n}\iprod{^{c}{\rm D}^{1-\alpha}
(\Pi\xi+\eta-\I^{1-\alpha} \Delta \eta) , \zeta'}\,dt\,.
\end{equation}
Now, using  the continuity property, Lemma \ref{lemma:positivity-equal-0} $(iii)$ (with $\epsilon=4$),  we notice that
\[
\begin{split}
\Big|\int_0^{t_n}\iprod{^{c}{\rm D}^{1-\alpha} \eta,\zeta'}\,dt\Big| &\le \frac{2}{c_\alpha^2}\int_0^{t_n} \iprod{^{c}{\rm D}^{1-\alpha}
\eta,\eta'}\,dt+\frac{1}{8}\int_0^{t_n}\iprod{^{c}{\rm D}^{1-\alpha} \zeta,\zeta'}\,dt,\\
 \Big| \int_0^{t_n}\iprod{^{c}{\rm D}^{1-\alpha} \Pi \xi,\zeta'}\,dt\Big| &\le
\frac{2}{c_\alpha^2}\int_0^{t_n} \iprod{^{c}{\rm D}^{1-\alpha} \Pi \xi,(\Pi \xi)'}\,dt+\frac{1}{8}\int_0^{t_n}\iprod{^{c}{\rm D}^{1-\alpha}
\zeta,\zeta'}\,dt,\end{split}\] and (with $\epsilon=2$),
\begin{align*} \Big|\int_0^{t_n}\iprod{^{c}{\rm D}^{1-\alpha}&(\I^{1-\alpha} \Delta \eta) ,\zeta'}\,dt\Big|\\
&\le \frac{1}{c_\alpha^2}\int_0^{t_n} \iprod{^{c}{\rm D}^{1-\alpha}(\I^{1-\alpha} \Delta \eta), (\I^{1-\alpha} \Delta \eta)'}\,dt+
\frac{1}{4}\int_0^{t_n}\iprod{^{c}{\rm
D}^{1-\alpha} \zeta,\zeta'}\,dt\\
&=\frac{1}{c_\alpha^2}\int_0^{t_n} \iprod{\Delta \eta,{^{c}{\rm D}}^{\alpha} \Delta \eta}\,dt+ \frac{1}{4}\int_0^{t_n}\iprod{^{c}{\rm
D}^{1-\alpha} \zeta,\zeta'}\,dt\,,
\end{align*}
where in the second inequality, we used the identity:
$$(\I^{1-\alpha} \Delta \eta)'(t)= \omega_{1-\alpha}(t)\,\Delta \eta(0)+{^{c}{\rm D}}^{\alpha} \Delta \eta(t)={^{c}{\rm D}}^{\alpha} \Delta \eta(t)\,.$$
Inserting the above inequalities in \eqref{eq:orthogonality} and rearranging the terms yield
\begin{multline}\label{eq:orthogonality2}  \int_0^{t_n}\iprod{^{c}{\rm D}^{1-\alpha} \zeta,\zeta'}\,dt+ \|\nabla \zeta(t_n)\|^{2} \\
\le  \frac{4}{c_\alpha^2}\int_0^{t_n}\left(\iprod{^{c}{\rm D}^{1-\alpha} \eta,\eta'}+ \iprod{\Delta \eta,{^{c}{\rm D}}^{\alpha} \Delta \eta}+
\iprod{^{c}{\rm D}^{1-\alpha} \Pi\xi,(\Pi\xi)'}\right)dt\,.
\end{multline}
To complete the proof, we still need to estimate the third term on the right-hand side of \eqref{eq:orthogonality2}.
%The first step towards this is to
%recall the following inequality from~\cite[Lemma 6.3]{LarssonThomeeWahlbin1998}: for $g \in L_2(0,t)$, there holds:  for $0<\alpha<1$,
%\begin{equation}
%\label{eq:LTW98} \begin{split}
% \int_0^t  \left(\int_0^s \omega_\alpha(s-w)\,  g(w)\,dw\right)^2\,ds &\le \omega_{\alpha+1}(t)
%\int_0^t\,\omega_{\alpha}(t-s) \int_0^s g^2(w)\,dw\,ds \\
%&  \le \omega_{\alpha+1}^2(t)  \int_0^t g^2(s)\,ds\,.
%\end{split}
%\end{equation}
By the Cauchy-Schwarz inequality, the inequality:
\[\Big(\int_0^{t_n}\|^{c}{\rm
D}^{1-\alpha} \Pi\xi\|^2dt\Big)^{1/2} \le \omega_{\alpha+1}(t_n)\Big(\int_0^{t_n}\| (\Pi\xi)'\|^2dt\Big)^{1/2},\]
 and  the property of the  operator $\Pi$ in \eqref{(ii)} (with  $\Pi\xi$ in place of $ \Pi u$), we have
\begin{align*}
\Big|\int_0^{t_n}\iprod{^{c}{\rm D}^{1-\alpha} \Pi\xi,\Pi\xi'}\,dt \Big|&\le  \int_0^{t_n}\|^{c}{\rm D}^{1-\alpha} \Pi\xi\|\, \|(\Pi \xi)'\|\,dt\\
 &\le \Big(\int_0^{t_n}\|^{c}{\rm
D}^{1-\alpha} \Pi\xi\|^2dt\Big)^{1/2}\Big(\int_0^{t_n}\| (\Pi\xi)'\|^2\,dt\Big)^{1/2} \\ & \le \omega_{\alpha+1}^2(t_n)\int_0^{t_n}\|
(\Pi\xi)'\|^2\,dt\le \omega_{\alpha+1}^2(t_n)\int_0^{t_n}\| \xi'\|^2\,dt.
\end{align*}
Finally, the desired inequality is obtained after inserting this estimate in \eqref{eq:orthogonality2}.
\end{proof}

\subsection{Regularity and time meshes}
As mentioned earlier,
%we want to remind the reader that the continuous solution $u$ of the time
%fractional model problem \eqref{eq:fractional equation} is not regular. More precisely,
the solution $u$ of the  fractional model problem \eqref{eq:fractional equation} has a singular behaviour near $t=0$. Under suitable regularity
assumptions on the initial data $u_0$ and the forcing term $f$ in problem \eqref{eq:fractional equation},  $u$ satisfies: for $t>0$ and for $1\le
q\le m+1,$
\begin{align} \| u^{(q)}(t)\|&\le c_q\, t^{\sigma-q}\quad {\rm and}\quad
\| \Delta u^{(q)}(t)\|\le d_q\, t^{\delta-q-1}  \label{eq:countable-regularity v1}
\end{align}
%\begin{equation} \label{eq:countable-regularity v2}
%\| u'(t)\|+t\| u''(t)\|\le C_2\, t^{\sigma-1},
%\end{equation}
for some positive constants $c_q$ and $d_q$, with $(1-\alpha)/2<\sigma<1$ and $\delta >1$. The proof of \eqref{eq:countable-regularity v1}
follows from the regularity analysis in \cite{McLean2010,McLeanMustapha2007}.

Because ~$u$ is not sufficiently smooth near~$t=0$, the global error in~$U_h$ fails to be $O(k^{m+1})$ accurate in time if we use a uniform
time step~$k$. %, especially for a relatively large $m$, that is, in the case of high-order DPG method.
 Typically, for high order methods over uniform time meshes,  one should not expect to observe global convergence rates of an order
 better than  $O(k^{\sigma})$ in the $L_\infty(0,T)$-norm. Now, to capture the singular behaviour of $u$ near $t=0$, following
\cite{McLeanMustapha2007,McLeanMustapha2009,MustaphaMcLean2012,MustaphaMcLean2013},  we  employ a family of non-uniform meshes that concentrate
the time levels near~$t=0$.  More precisely, we assume that for a fixed parameter $\gamma\ge1$,
\begin{equation}\label{eq: tn standard}
t_n=(n k)^\gamma \quad\text{with~~$k=\frac{T^{1/\gamma}}{N}$~~for $0\le n \le N$.}
\end{equation}
 Noting that the time step sizes are nondecreasing, that is, $k_i\le k_j$ for $i\le j$. For $2\le n\le N$ one can show that
%\begin{equation}\label{eq:k1 mesh}
%c_{\gamma}k^{\gamma}\leq k_{1}\leq C_{\gamma}k^{\gamma}\quad {\rm where}~~k=\max_{n=1}^N k_n,
%\end{equation}
%and
\begin{equation}\label{eq:kn mesh} \frac{\gamma}{2^{\gamma-1}}k\, {t_{n}}^{1-\frac {1}{\gamma}}\leq {k_{n}}\leq
\gamma\,k\, {t_{n}}^{1-\frac {1}{\gamma}} ~~{\rm and}~~t_{n}\leq 2^\gamma {t_{n-1}}.
\end{equation}

The aim now is to bound the first and second terms on the right-hand side of \eqref{eq: zeta estimate} in Theorem \ref{thm: zeta estimate}.
\subsection{Estimate of $\int_0^{t_n}\iprod{{^c}{\rm D}^{1-\alpha} \eta,\eta'}\,dt$}
{ Assume that $u$ satisfies the first regularity assumption in \eqref{eq:countable-regularity v1}. Then, there exists a positive constant $C$ that depends on $d_1$, $\sigma,$ $\alpha$, $\gamma$, $m$ and $T$,
such that, for $1\le n\le N,$
\[
\int_0^{t_n}\iprod{{^c}{\rm D}^{1-\alpha} \eta,\eta'}\,dt\le C\, k^{2m+\alpha}\quad{\rm for}\quad \gamma \ge (2m+1+\alpha)/(\alpha+2\sigma-1)\,.\]
\begin{proof} We start our proof by splitting $\int_0^{t_n}\iprod{{^c}{\rm D}^{1-\alpha} \eta,\eta'}\,dt$ as follows:
\begin{equation}\label{eq: split}
 \int_0^{t_n}\iprod{{^c}{\rm D}^{1-\alpha} \eta,\eta'}\,dt
=\int_{I_1}
\iprod{\A^\alpha_{1}(t),\eta'(t)}\,dt+\sum_{j=2}^n\int_{I_j}\iprod{\A^\alpha_2(t)+\A^\alpha_{3,j}(t),\eta'(t)}\,dt,\end{equation} where
\[\begin{split}
\A^\alpha_{1}(t)&:=\int_0^{t} \omega_{\alpha}(t-s)\,\eta'(s)\,ds\quad {\rm and}\\
\A^\alpha_{2}(t)&:=\int_0^{t_1} \omega_{\alpha}(t-s)\,\eta'(s)\,ds\quad \\
\A^\alpha_{3,j}(t)&:=\int_{t_1}^{t} \omega_{\alpha}(t-s)  \,\eta'(s)\,ds\\
=-&\int_{t_1}^{t_{j-1}} \omega_{\alpha-1}(t-s)  \,\eta(s)\,ds+\int_{t_{j-1}}^t \omega_{\alpha}(t-s)
\,\eta'(s)\,ds\quad {\rm for}~~ t \in I_j~{\rm with}~j\ge 2\,.
%\\
% \A^\alpha_{4,j}(t)&:=\int_{t_{j-1}}^t \omega_{\alpha}(t-s)
%\,\eta'(s)\,ds\quad {\rm for}~~ j\ge 2\,.
\end{split}\] For $t\in I_1$, from \eqref{eq:estimate eta} (with $n=1$) and the first regularity property in
\eqref{eq:countable-regularity v1} ($\sigma<1$), we observe \begin{equation}\label{eq: initial estimate}\begin{aligned} \|\eta'(t)\|&\le
\frac{1}{t_1}\int_0^{t_1}\|u'(s)\|ds+\|u'(t)\|\\
&\le \frac{C}{t_1}\int_0^{t_1} s^{\sigma-1}\,ds+Ct^{\sigma-1}\le C(t_1^{\sigma-1}+ t^{\sigma-1})\le Ct^{\sigma-1}\,.\end{aligned}\end{equation}
Hence, using the Cauchy-Schwarz inequality and integrating, we have
 \begin{equation}\label{eq: A1t}
\begin{aligned}
\int_{I_1} |\iprod{\A^\alpha_{1}(t),\eta'(t)}|\,dt &\le \int_{I_1} \|\eta'(t)\|\int_0^{t}
\omega_\alpha(t-s)\,\| \eta'(s)\|\,ds\, dt\\
&\le C\,\int_{I_1} t^{\sigma-1} \int_0^{t} \frac{(t-s)^{\alpha-1}}{\Gamma(\alpha)}\, s^{\sigma-1}\,ds\, dt\\
&= C\,\frac{\Gamma(\sigma)}{\Gamma(\sigma+\alpha)}\int_{I_1} t^{\sigma-1}\, t^{\sigma+\alpha-1}\, dt\le   C t_1^{\alpha+2\sigma-1}\,.
\end{aligned}\end{equation}
To estimate the term involving $\A^\alpha_2(t)$, we use \eqref{eq: initial estimate},   and the inequality: $(t_j-s)^\alpha-(t_{j-1}-s)^\alpha\le k_j^\alpha$ after integrating,
\begin{align*}
\int_{I_j} |\iprod{\A^\alpha_{2}(t),\eta'(t)}|\,dt &\le \int_{I_j} \|\eta'(t)\|\int_{I_1}
\omega_\alpha(t-s)\,\| \eta'(s)\|\,ds\, dt\\
&\le C\,\|\eta'\|_{I_j}\int_{I_1} \int_{I_j}  (t-s)^{\alpha-1}\, s^{\sigma-1}\,dt\, ds\\
 &\le     C\,\|\eta'\|_{I_j} k_j^\alpha\, t_1^{\sigma}
=     C\,k_j \|\eta'\|_{I_j}  k_j^{\alpha-1} t_1^{\sigma}\quad{\rm for}~~j\ge 2\,.
 %\\
 %&\le     C\,k_j^{m} \|u^{(m+1)}\|_{I_j} k_j^\alpha\, t_1^{\sigma}
\end{align*}
Hence, summing over $j$, and using $k_j^{(\alpha-1)/2} \le  k_1^{(\alpha-1)/2} $ for $j\ge 1$ (from the mesh properties),  we get
 \begin{equation}\label{eq: A2t}
\begin{aligned}
\sum_{j=2}^n \int_{I_j} |\iprod{\A^\alpha_{2}(t),\eta'(t)}|\,dt &\le  C\, \sum_{j=2}^n ( k_j^{(\alpha+1)/2}\|\eta'\|_{I_j})t_1^{\sigma+(\alpha-1)/2}
\\
&\le C\,\left( \max_{j=2}^n ( k_j^{\alpha+1}\|\eta'\|_{I_j}^2) +  t_1^{\alpha+2\sigma-1}\right)k^{-1}\,.
\end{aligned}\end{equation}
%To estimate the term involving $\A^\alpha_{3,j}(t)$, we integrate by parts to observe
%\begin{align*}
%\A^\alpha_{3,j}(t) &= -\int_{t_1}^{t_{j-1}} \omega_{\alpha-1}(t-s) \eta(s)\,ds\end{align*} and then,
It remains to estimate the term  $\int_{I_j} \iprod{\A^\alpha_{3,j}(t),\eta'(t)}\,dt$. Splitting it into two terms, changing the order of integrals and using; $ \int_{t_1}^{t_{j-1}} [\omega_{\alpha}(t_{j-1}-s)-\omega_{\alpha}(t_{j}-s)] \,ds \le \omega_{\alpha+1}(k_j)$, we notice that
\begin{align*}
\int_{I_j} \|\A^\alpha_{3,j}(t)\|\,dt &\le    \int_{I_j}\Big( \int_{t_1}^{t_{j-1}}
\omega_{\alpha-1}(t-s)\,\|\eta(s)\|
\,ds+ \int_{t_{j-1}}^t
\omega_{\alpha}(t-s)\,\|\eta'(s)\|
\,ds\Big) \,dt\\
&\le  \int_{t_1}^{t_{j-1}} [\omega_{\alpha}(t_{j-1}-s)-\omega_{\alpha}(t_{j}-s)]\,\|\eta(s)\| \,ds+\omega_{\alpha+2}(k_j)\,\|\eta'\|_{I_j}
\\
&\le   \omega_{\alpha+1}(k_j)\max_{i=2}^{j-1}\|\eta\|_{I_i}+\omega_{\alpha+2}(k_j)\,\|\eta'\|_{I_j} \,.
%\le  C\,k_j^\alpha\max_{i=2}^j(k_i\|\eta'\|_{I_i}),
\end{align*}
%where in the last inequality we used that $\eta(t)=\int_{t_{i-1}}^t \eta'(s)\,ds$ (which follows from the definition of the projection $\Pi$).
%Similarly,
%\begin{align*} \int_{I_j} \|\A^\alpha_{4,j}(t)\|\,dt \le
% \|\eta'\|_{I_j} \int_{I_j}
%  \int_{t_{j-1}}^t \omega_{\alpha}(t-s)\,ds\,dt \le  \|\eta'\|_{I_j}
% \omega_{\alpha+2}(k_j) \,.
% \end{align*}
 Therefore,  summing over $j$,  then, using the Cauchy-Schwarz inequality, the inequality $k_j^{\alpha-1}\le k_i^{\alpha-1}$ for $i\le j$, and the identity $\eta(t)=\int_{t_{i-1}}^t \eta'(s)\,ds$ for $t\in I_i$ (because $\eta(t_i)=0$ for $i=0,1,\cdots, N$),    we observe
 \begin{equation}\label{ A3t partial}\begin{aligned}  \int_{t_1}^{t_n}
|\iprod{\A^\alpha_{3,j}(t),\eta'(t)}|\,dt&\le  C\sum_{j=2}^n  \left(k_j^{\frac{\alpha+1}{2}}\|\eta'\|_{I_j}
k_j^{\frac{\alpha-1}{2}}\max_{i=2}^{j-1} \|\eta\|_{I_i}+ k_j^{\alpha+1}\|\eta'\|_{I_j}^2\right)
\\
&\le  C\sum_{j=2}^n  \left(k_j^{\alpha-1}
\max_{i=2}^{j-1} \|\eta\|_{I_i}^2+ k_j^{\alpha+1}\|\eta'\|_{I_j}^2\right)
\\&\le  C\,k^{-1} \max_{j=2}^n (k_j^{\alpha+1}\|\eta'\|_{I_j}^2)\,.
\end{aligned}
\end{equation}
Now, from the interpolation errors in  Theorem \ref{thm: estimate of eta}, the first regularity assumption in \eqref{eq:countable-regularity v1}, the time mesh property \eqref{eq:kn mesh}, and the graded exponent time mesh assumption, $\gamma \ge (2m+1+\alpha)/(2\sigma+\alpha-1),$  we get
\begin{equation}\label{eq: estimate of eta_n}
\begin{aligned}
 k_j^{\alpha+1}\|\eta'\|_{I_j}^2  &\le C\,
 k_j^{2m+\alpha}\int_{I_j} \|u^{(m+1)}(t)\|^2\,dt
\\
& \le C\,
 k_j^{2m+1+\alpha}t_j^{2(\sigma-m-1)}
\\&\le   C\,k^{2m+1+\alpha}  t_j^{2m+1+\alpha-(2m+1+\alpha)/\gamma} \, t_j^{2(\sigma-m-1)}
\\&=  C\,k^{2m+1+\alpha}   \, t_{j}^{2\sigma+\alpha-1-(2m+1+\alpha)/\gamma}
\le   C\,k^{2m+1+\alpha} \,.
\end{aligned}
\end{equation}
Finally, combining  \eqref{eq:
split}--\eqref{ A3t partial},  the above bound, and using the  inequality  $t_1^{2\sigma+\alpha-1}=k^{\gamma(2\sigma+\alpha-1)}$ $\le k^{m+\alpha+1}$,    yield the desired estimate. \end{proof}
\subsection{Estimate of $\int_0^{t_n} \iprod{{^c}{\rm D}^{\alpha} \Delta \eta,\Delta \eta}\,dt$}
Assume that $u$ satisfies the second regularity assumption in \eqref{eq:countable-regularity v1}. Then, for $1\le n\le N,$ we have
\begin{equation}\label{eq: split 2}
\int_0^{t_n}\iprod{{^c}{\rm D}^{\alpha} \Delta \eta,\Delta \eta}\,dt\le C\, k^{2m+1}~{\rm for}~ \gamma \ge (m+1)/(\delta-1)\,,\end{equation} where the constant $C$ depends on $d_2$, $\delta,$
$\alpha$, $\gamma$, $m$ and $T$\,.
\begin{proof}
Following the decomposition in \eqref{eq: split},
\begin{equation}\label{eq: Delta split}
\begin{aligned}
\int_0^{t_n}\iprod{{^c}{\rm D}^{\alpha} \Delta \eta,\Delta \eta}\,dt=\int_{I_1} \iprod{\Delta &\A^{1-\alpha}_{1},\Delta
\eta}\,dt
+\int_{t_1}^{t_n}\iprod{\Delta(\A^{1-\alpha}_2+\A^{1-\alpha}_{3,j}),\Delta \eta}\,dt\,. \end{aligned}\end{equation}
To bound the first term, we use
 $\Delta
\A^{1-\alpha}_1(t)=\frac{\partial}{\partial t} \int_0^{t} \omega_{1-\alpha}(t-s)\,\Delta \eta(s)\,ds$ for $t\in I_1$ (because $\eta(0)=0$) and then, integrating by parts and using
the interpolation properties of the projection operatzor $\Pi$, yield
\begin{align*}
\int_{I_1} \iprod{\Delta \A^{1-\alpha}_{1}(t),\Delta \eta(t)}\,dt
 &= \int_{I_1} \iprod{\Delta \eta'(t),\I^{1-\alpha}\Delta \eta(t)}\, dt
=\int_{I_1} \iprod{\Delta \eta'(t),\I^{2-\alpha}\Delta \eta'(t)}\, dt\,.
\end{align*}
 Following the steps in \eqref{eq: initial estimate} and using the second regularity property in \eqref{eq:countable-regularity v1}, we obtain
 \[\begin{aligned} \|\Delta \eta'(t)\|&\le
\frac{1}{t_1}\int_0^{t_1}\|\Delta u'(s)\|ds+\|\Delta u'(t)\|\\
&\le \frac{C}{t_1}\int_0^{t_1} s^{\delta-2}\,ds+Ct^{\delta-2}\le C(t_1^{\delta-2}+ t^{\delta-2})\le C\begin{cases} t^{\delta-2}~~~{\rm for}~~1<\delta<2\\
t_1^{\delta-2}~~~{\rm for}~~\delta\ge 2\,.\end{cases}\end{aligned}\]
  Hence, an application of the Cauchy-Schwarz inequality followed by direct integrations, yield
 \begin{equation}
\begin{aligned}\label{eq: Delta a1t}
\int_{I_1} |\iprod{\Delta& \A^{1-\alpha}_{1}(t),\Delta \eta(t)}|\,dt \le
 \int_{I_1} \|\Delta \eta'(t)\|\int_0^t
\omega_{2-\alpha}(t-s)\,\|\Delta  \eta'(s)\|\,ds\, dt\\
& \le C \int_{I_1} t^{\delta-2} \int_0^{t} (t-s)^{1-\alpha}\,s^{\delta-2}\,ds\, dt \le C t_1^{2\delta-\alpha-1}\le C t_1^{2(\delta-1)}
\end{aligned}\end{equation}
 for $1<\delta<2$, where in the last inequality, we used $t_1^{-\alpha}\le C\,t_1^{-1}$ since $0<\alpha<1$. A similar bound can be achieved when $\delta \ge 2.$

 In a similar manner (see the steps used to obtain \eqref{eq: A2t})
\begin{align*}
\int_{I_j} |\iprod{\Delta \A^{1-\alpha}_{2}(t),&\Delta \eta(t)}|\,dt
 \le \int_{I_j} \|\Delta \eta(t)\|\int_0^{t_1}
\omega_{1-\alpha}(t-s)\,\|\Delta  \eta'(s)\|\,ds\, dt\\
&  \le C\, \|\Delta \eta\|_{I_j}\int_{I_j} \int_0^{t_1}
\omega_{1-\alpha}(t-s)\,s^{\delta-2}\,ds\, dt\\
&  = C\, k_j \|\Delta \eta'\|_{I_j} \int_0^{t_1}
[\omega_{2-\alpha}(t_j-s)-\omega_{2-\alpha}(t_{j-1}-s)]\,s^{\delta-2}\,ds\\
&  = C\,k_j \omega_{2-\alpha}(k_j) \|\Delta \eta'\|_{I_j} \int_0^{t_1} \,s^{\delta-2}\,ds
\le  C\, k_j^{2-\alpha} \|\Delta \eta'\|_{I_j}t_1^{\delta-1}\end{align*}
where the identity $\Delta \eta(t)=\int_{t_{j-1}}^t \Delta \eta'(s)\,ds$ is also used.  Thus,
\begin{equation}\label{eq: Delta A2t n}
 \sum_{j=2}^n \int_{I_n} |\iprod{\Delta \A^{1-\alpha}_{2}(t),\Delta \eta(t)}|\,dt
   \le  C\,  \left(\max_{j=2}^n (k_j^2\|\Delta \eta'\|_{I_j}^2)+t_1^{2(\delta-1)}\right)\sum_{j=2}^n k_j^{1-\alpha} \,.
\end{equation}
It remains to estimate the term $\int_{t_1}^{t_n} |\iprod{\Delta \A^{1-\alpha}_{3,j},\Delta \eta}|\,dt$. Following the steps  in \eqref{ A3t partial},
\begin{align*}
\sum_{j=2}^n \int_{I_j} |\iprod{\Delta \A^{1-\alpha}_{3,j},\Delta \eta}|\,dt
 &\le  C\sum_{j=2}^n  \left(k_j^{1-\alpha}\|\Delta \eta\|_{I_j}
\max_{i=2}^{j-1} \|\Delta \eta\|_{I_i}+ k_j^{2-\alpha}\|\Delta \eta'\|_{I_j} \|\Delta \eta\|_{I_j}\right)
\\
&\le  C\sum_{j=2}^n k_J^{1-\alpha} \left(k_j\|\Delta \eta'\|_{I_j}
\max_{i=2}^{j-1} \|\Delta \eta\|_{I_i}+ k_j^2\|\Delta \eta'\|_{I_j}^2\right)
\\
&\le  C\sum_{j=2}^n  k_j^{1-\alpha} \left(\max_{i=2}^{j-1} \|\Delta \eta\|_{I_i}^2+ k_j^2\|\Delta \eta'\|_{I_j}^2\right)
\\
& \le  C
\max_{j=2}^{n} ( k_j^2\|\Delta \eta'\|_{I_j}^2)\sum_{j=2}^nk_j^{1-\alpha}\,.
\end{align*}
 Therefore, the desired estimate in
\eqref{eq: split 2} follows from \eqref{eq: Delta split}, \eqref{eq: Delta a1t}, \eqref{eq: Delta A2t n}, the above bound, and the following two inequalities: $\sum_{j=2}^nk_j^{1-\alpha}\le C\,k^{-\alpha}$, $t_1^{\delta -1}=k^{\gamma(\delta-1)}\le k^{m+1}$ (by the mesh assumption $\gamma \ge  (m+1)/(\delta-1)$) and
the estimate \begin{multline*}
k_j\|\Delta \eta'\|_{I_j}  \le C\,  k_j^{m+1}\, \|\Delta u^{(m+1)}\|_{I_j}
 \le C\,  k_j^{m+1}\, t_j^{\delta-m-2}\\
 \le C\,  k^{m+1}\,t_j^{m+1-(m+1)/\gamma} t_j^{\delta-m-2}  = C\, k^{m+1} t_j^{\delta-1-\frac{m+1}{\gamma}}\le C\,k^{m+1}\,.
\end{multline*}
Here, we used Theorem \ref{thm: estimate of eta}, the second regularity assumption in \eqref{eq:countable-regularity v1}, the mesh property \eqref{eq:kn mesh}, and the mesh assumption $\gamma \ge  (m+1)/(\delta-1)$.
 \end{proof}}
%%%%%%%%%%%%%%%%%%%%%%%%%%%%%%%%%%%%%%%%%%%%%%%%%%%%%%
\subsection{The error estimates}
We are now ready to obtain our main error convergence  results for the DPG-FE solution. In the next theorem, we derive suboptimal algebraic
rates of convergence in time (short by order $1-\alpha/2$ from being optimal), and optimal convergence rates in the spacial discretization
provided the solution $u$ of \eqref{eq:fractional equation} is sufficiently regular. However, our numerical results illustrate an optimal
convergence rate in both time and space.
%  To derive error estimates in the norm $\|\cdot\|_{I_n}$, in addition to the use of the achieved
%bounds,we shall make use of the following inverse estimate from \cite[Lemma~3.1]{SchoetzauSchwabDGODE}.
%\begin{lemma}\label{lemma: uniform bound}
%Let $\phi \in \W(S_h)$. Then for $1\le n\le N$,  we have
%\[\|\phi\|^2_{I_n} \le
%C\left(\log(m+2)\int_{I_n}\|\phi'\|^2(t-t_{n-1})\,dt+\|\phi(t_n)\|^2\right).\]
%\end{lemma}

%Therefore, squaring both sides  and using the geometric-arithmetic  mean inequality will complete the proof.

{\begin{theorem}\label{thm: main results} Let $u_0\in H^{r+1}(\Omega)$,  $f \in H^1((0,T);H^2(\Omega)),$ and let  the solution $u\in W^{1,1}((0,T);H^2(\Omega))$ of problem~\eqref{eq:fractional equation} satisfy the regularity properties
 in \eqref{eq:countable-regularity v1} (with $\sigma>(1-\alpha)/2$
 and  $\delta >1$). Moreover,  we assume that  $ u(t_1)\in H^{r+1}(\Omega)$
 and $u' \in L_2((t_1,T);H^{r+1}(\Omega))$. Let $U_h\in \W(S_h)$ be the DPG-FE approximation defined by
\eqref{eq: DPGFM}, and  assume that the time mesh graded factor $ \gamma \ge  \max
\Big\{\frac{m+1}{\delta-1},\frac{2m+1+\alpha}{2\sigma+\alpha-1}\Big\} $. Then,
\begin{multline*}
\max_{n=1}^n \|U_h-u\|_{I_n}^2  \le C k^{2m+\alpha}
+C\,h^{2r+2}\Big(\|u_0\|_{r+1}^2+\|u(t_1)\|_{r+1}^2+ \int_{t_1}^{t_n} \|u'(t)\|_{r+1}^2\,dt\Big),\end{multline*} where $C$ is a constant that depends
only on $d_1$, $d_2$, $\alpha$, $\sigma$, $\delta$, $\gamma$, $m$ and $T$ \,.
\end{theorem}
\begin{proof} From the decomposition: $u-U_h=\zeta+\Pi \xi + \eta$ (given in \eqref{eq: decompose Uh -u}),
\begin{equation}\label{eq: decomposition norm}\|u-U_h\|_{I_n}
\le \|\zeta\|_{I_n} +\|\Pi \xi\|_{I_n} +\|\eta\|_{I_n}\quad{\rm for}~~~ 1\le n\le N\,.\end{equation}

{\bf Step 1:} Estimating $\|\zeta\|_{I_n}$. Since $\zeta(0)=0$, we notice that $\zeta(t)=\int_0^t\zeta'(s)\,ds=\I^{1-\frac{ \alpha}{2}} (^c{\rm
D}^{1-\frac{\alpha}{2}} \zeta)(t)\,. $ So, for any $t\in I_n$,  the use of the Cauchy-Schwarz inequality yields
\[
\|\zeta(t)\|^2 \leq \Big(\int_0^t\, \omega_{1-\frac{\alpha}{2}}(t-s)\, \|^c{\rm D}^{1-\frac{\alpha}{2}} \zeta(s)\|\,ds\Big)^2 \leq \int_0^t\,
\omega^2_{1-\frac{\alpha}{2}}(s)\,ds \int_0^t\,\|^c{\rm D}^{1-\frac{\alpha}{2}} \zeta(s)\|^{2}\,ds
\]
and hence, by property $(ii)$ in   lemma \ref{lemma:positivity-equal-0},
 Theorem  \ref{thm: zeta estimate}, and the achieved bounds in \eqref{eq: split} and \eqref{eq: split 2},  we find that
\begin{align*}
\| \zeta\|_{I_n}^2 &\le C\,k^{2m+\alpha}\quad{\rm for}~~~ 1\le n\le N\,.
\end{align*}

{\bf Step 2:} Estimating $\|\Pi \xi\|_{I_n}$.
For $n\ge 2,$ $\Pi \xi(t)=\int_{t_1}^{t} (\Pi \xi)'(s)\,ds + (\Pi \xi)(t_1)$ for $t \in I_n.$ Since, $(\Pi \xi)(t_1)=\xi(t_1)$, an application of the Cauchy-Schwarz inequality gives
\begin{align*}
\|\Pi \xi\|_{I_n}^2 &\le  \left(\int_{t_1}^{t_n}\|(\Pi \xi)'(t)\|\,dt+\|\xi(t_1)\|\right)^2 \le 4 \int_{t_1}^{t_n}\|
\Pi \xi'(t)\|^{2}dt+2\|\xi(t_1)\|^2\,.
\end{align*}
For $n=1$, $\Pi \xi$ is linear in the time variable t. So, $|\Pi \xi\|_{I_1}\le \max \{\|\Pi
\xi(0)\|,\|\Pi \xi(t_1)\|\}$\,. Thus,  by \eqref{(ii)} with $\xi$ in place of $u$ and the Ritz projection  approximation error \eqref{eq:estimate ritz projector},
\begin{align*}
\|\Pi \xi\|_{I_n}^2 &  \le 4 \int_{t_1}^{t_n}\|
 \xi'(t)\|^{2}dt+2\|\xi(t_1)\|^2+\|\xi(0)\|^2\\
&  \le C\,h^{2r+2}\Big(\int_{t_1}^{t_n} \|u'(t)\|_{r+1}^2\,dt+\|u(t_1)\|_{r+1}^2+\|u_0\|_{r+1}^2\Big)\quad{\rm for}~~1\le n\le N\,.
\end{align*}
{\bf Step 3:} Estimating $\| \eta\|_{I_n}.$ For $t\in I_n,$ $\eta(t)=\int_{t_{n-1}}^t\eta'(s)\,ds$. Thus,
\begin{align*}
\| \eta\|_{I_n} &\le  \begin{cases} \int_0^{t_1}\|\eta'(t)\|\,dt \le C\,t_1^{\sigma} \le C\,t_1^{\sigma+\frac{\alpha-1}{2}} \le C \,k^{m+\frac{\alpha+1}{2}},\\
k_n\,\| \eta'\|_{I_n} \le C\,k_n^{\frac{\alpha+1}{2}}\| \eta'\|_{I_n}\le C\,k^{m+\frac{\alpha}{2}} \quad{\rm for}~~~ 2\le n\le N\end{cases}
\end{align*}
where we used  \eqref{eq: initial estimate} and the mesh assumption $ \gamma \ge  \frac{2m+1+\alpha}{2\sigma+\alpha-1}$ in the first estimate and \eqref{eq: estimate of eta_n} in the second one.

 Therefore, the desired
 error estimate  follows from the decomposition \eqref{eq: decomposition norm} and the bounds in {\bf Step 1}--{\bf Step 3}. \end{proof}}

%%%%%%%%%%%%%%%%%%%%%%%%%%%%%%%%%%%%%%%
\section{Numerical results}\label{sec:implementation and numerics}
 \setcounter{equation}{0} \setcounter{theorem}{0}
%%%%%%%%%%%%%%%%%%%%%%%%%%%%%%%%%%%%%%%%%%%%
 In this section, we demonstrate the validity of the derived error results when  $\Delta u=u_{xx}$,  $\Omega=(0,1)$  and $T=1$ in the time-fractional
  problem  \eqref{eq:fractional equation}.  To evaluate the errors, we
introduce the finer grid
\begin{equation}\label{eq: fine grid}
\G^{q}=\{\,t_{j-1}+ n k_j/q\,:\, 1\leq j \leq N,\ 0 \leq n\leq q\,\}
\end{equation}
(recall that, $N$ is the number of time mesh subintervals). Thus, for large values of~$q$, the error measure $
|||v|||_{q}:=\max_{t\in\G^{q}}\|v(t)\|$ approximate  the norm $\|v\|_{L_\infty((0,T);L_2(\Omega))}$. To compute the spatial $L_2$-norm, we apply
a composite Gauss quadrature rule with $(r+1)$-points on each interval of the finest spatial mesh where $r$ is the degree of the approximate
solution in the spatial variable.

\subsection{Example 1} We choose  $u_0$ and $f$  such that the exact solution is $ u(x,t) = t^{\alpha+1}{\sin}(\pi x).$ It can be seen that the regularity conditions in \eqref{eq:countable-regularity v1}  hold for $\sigma=\alpha+1$ and $\delta=\alpha+2$.

First, to test the accuracy of the DPG-FE scheme \eqref{eq: DPGFM} (with degree $m$ in the time variable and $r$ in the spatial variable) on the
non-uniformly time graded meshes in~\eqref{eq: tn standard} for various choices of~$\gamma\ge1$,
 $h$ (the spatial step-size) will be chosen such that the temporal errors are dominating. Thus, from Theorem \ref{thm: main results}, we expect to observe convergence of
order $O(k^{m+\alpha/2})$ for $ \gamma \ge \max \Big\{\frac{m+1}{\alpha+1}, \frac{2m+1+\alpha}{3\alpha+1}\Big\} $. However, the numerical results in
Table \ref{tab: ||U_h-u|| alpha=0.2 Example 1 new} illustrate more optimistic convergence rates.  We observe a uniform global error bounded by
 $Ck^{\min\{\gamma(\alpha+1),m+1\}}$ for $\gamma \ge 1$,  which is optimal for
$\gamma\ge (m+1)/(\alpha+1)$. So, the numerical results  also demonstrated that the grading mesh parameter $\gamma$ is relaxed. The results are
also displayed graphically in Figure~\ref{fig1: DPG error Example 1}, where we show the errors against the number of time subintervals $N$, in
the semi-logarithmic scale.
  %%%%%%%%%%%%%%%%%%%%%%%%%%%%%%%%%%%%%%%%%%%%%%%%%%%%%%%%%%%%%%%%%%%%%
 \begin{figure}
 \begin{center}
 \scalebox{0.36}{\includegraphics{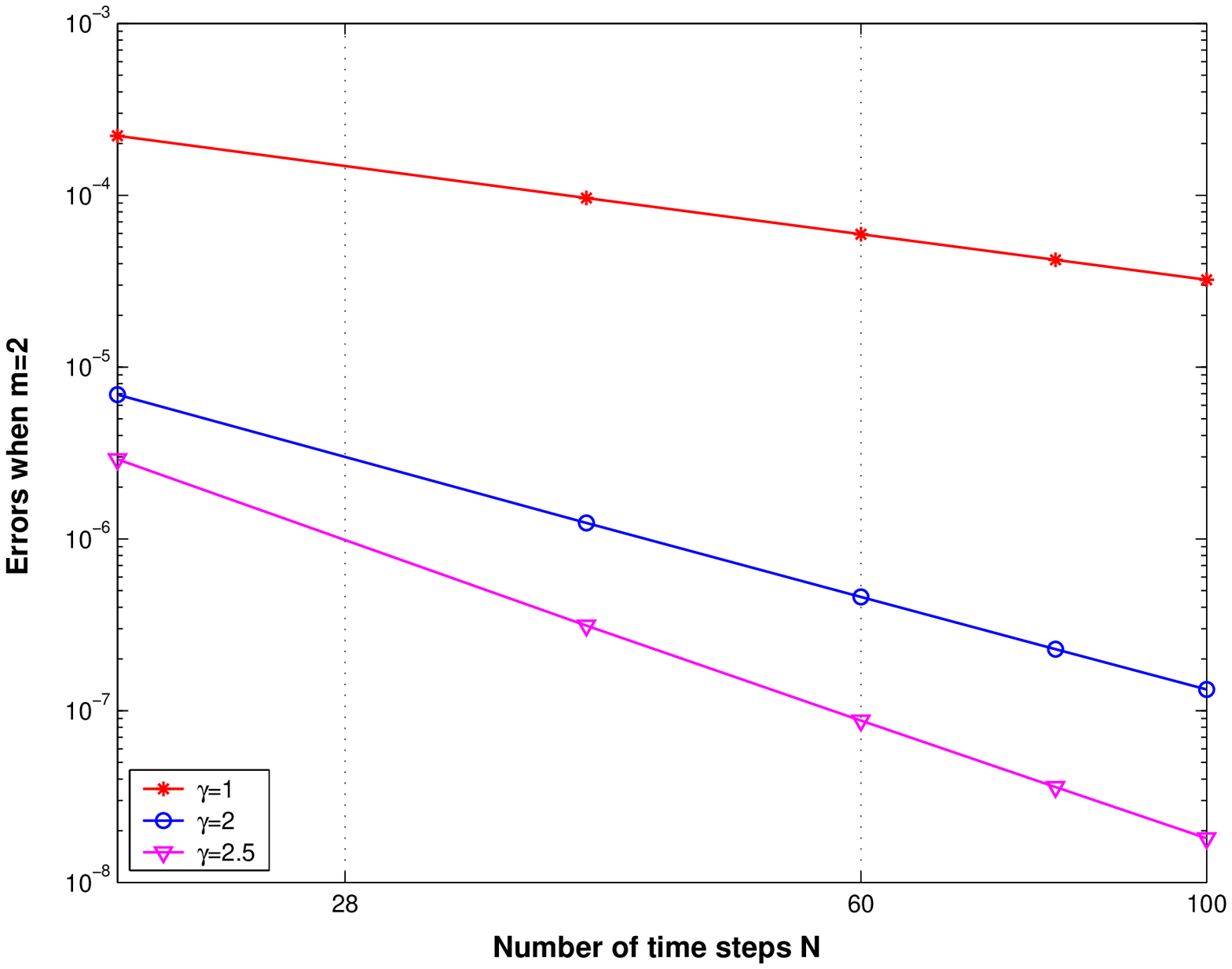} \includegraphics{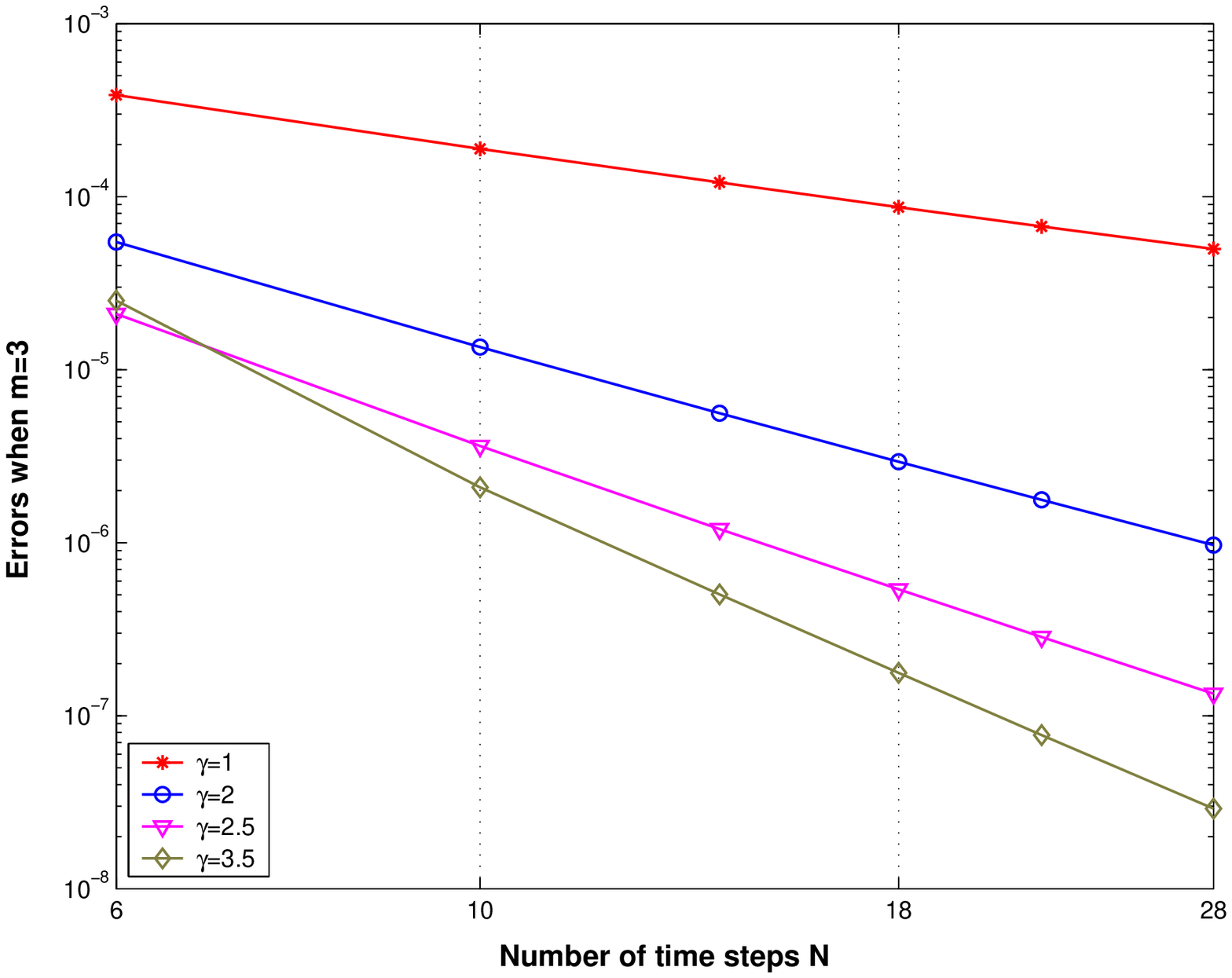}}
  \caption{The errors $|||U_h-u|||_{10}$ for Example 1  plotted against $N$ for different choices of~$\gamma$ and for $m=2,\,3$,
 with $\alpha=0.2$.}\label{fig1: DPG error Example 1}
 \end{center}
 \end{figure}
%%%%%%%%%%%%%%%%%%%%%%%%%%%%%%%%%%%%%%%%%%%%%%%%%%%%%%%%%%%%%%%%%%%%
\begin{table}[htb]
\renewcommand{\arraystretch}{1.0}
\begin{center}
\begin{tabular}{|r|rr|rr|rr|rr|}
\hline \multicolumn{9}{|c|}{$m=1$}\\
 \hline $N$ &\multicolumn{2}{c|}{$\gamma=1$} &\multicolumn{2}{c|}{$\gamma=1.4$}&\multicolumn{2}{c|}{$\gamma=1.8$}&\multicolumn{2}{c|}{}
\\ \hline
   20& 9.83e-04&     & 2.49e-04&     & 2.67e-04&     & &     \\
   40& 4.45e-04& 1.14& 8.02e-05& 1.64& 6.66e-05& 1.99& &  \\
   80& 2.01e-04& 1.16& 2.55e-05& 1.65& 1.66e-05& 2.00& &    \\
  160& 8.91e-05& 1.17& 8.04e-06& 1.67& 4.01e-06& 2.05& &   \\
 \hline
 \hline \multicolumn{9}{|c|}{$m=2$}\\
 \hline $N$ &\multicolumn{2}{c|}{$\gamma=1$} &\multicolumn{2}{c|}{$\gamma=2$}&\multicolumn{2}{c|}{$\gamma=2.5$}&\multicolumn{2}{c|}{}
\\ \hline
  20& 2.22e-04&     & 6.92e-06&     & 2.91e-06&      & & \\
  40& 9.64e-05& 1.20& 1.24e-06& 2.48& 3.13e-07& 3.21 & & \\
  60& 5.93e-05& 1.20& 4.59e-07& 2.45& 8.77e-08& 3.14 & & \\
  80& 4.21e-05& 1.19& 2.28e-07& 2.43& 3.59e-08& 3.10 & & \\
 100& 3.22e-05& 1.19& 1.33e-07& 2.43& 1.81e-08& 3.08 & & \\
 \hline
 \hline \multicolumn{9}{|c|}{$m=3$}\\
 \hline $N$ &\multicolumn{2}{c|}{$\gamma=1$} &\multicolumn{2}{c|}{$\gamma=2$}&\multicolumn{2}{c|}{$\gamma=2.5$}&\multicolumn{2}{c|}{$\gamma=3.5$}
\\ \hline
% 6& 3.87e-04&     & 5.47e-05&     & 2.10e-05&     & 2.51e-05&     \\
 10& 1.89e-04& 1.40& 1.35e-05& 2.74& 3.62e-06& 3.44& 2.09e-06& 4.87\\
 14& 1.21e-04& 1.34& 5.60e-06& 2.62& 1.20e-06& 3.28& 5.02e-07& 4.23\\
 18& 8.69e-05& 1.31& 2.94e-06& 2.56& 5.38e-07& 3.20& 1.77e-07& 4.15\\
 22& 6.72e-05& 1.28& 1.77e-06& 2.53& 2.85e-07& 3.16& 7.74e-08& 4.12\\
 28& 4.99e-05& 1.23& 9.69e-07& 2.50& 1.34e-07& 3.13& 2.91e-08& 4.05\\
 40& 3.23e-05& 1.22& 4.01e-07& 2.47& 4.45e-08& 3.08& & \\
\hline
  \end{tabular} \vspace{0.05in} \caption {The errors $|||U_h-u|||_{10}$
 for different time mesh gradings  with $\alpha=0.2$. We observe convergence of order $k^{(\alpha+1)\gamma} (=k^{1.2\gamma})$ for $1\le
\gamma \le (m+1)/(\alpha+1)$ for $m=1,\,2,\,3$.} \label{tab: ||U_h-u|| alpha=0.2 Example 1 new}
\end{center}
\end{table}
In Figure~\ref{fig: error time against alpha}, we demonstrate the positive influence of time graded mesh power $\gamma$ on the error  that
remains valid for different values of $0.1\le \alpha\le 0.9$. The errors achieved as a function of $\alpha$ for different values of $\gamma$,
but for a fixed $N=60$ and a fixed $m=2.$

%%%%%%%%%%%%%%%%%%%%%%%%%%%%%%%%%%%%%%%%%%%%%%%%%%%%%%%%%%%%%%%%%
\begin{figure}[htb]
\begin{center}
\scalebox{0.38}{\includegraphics{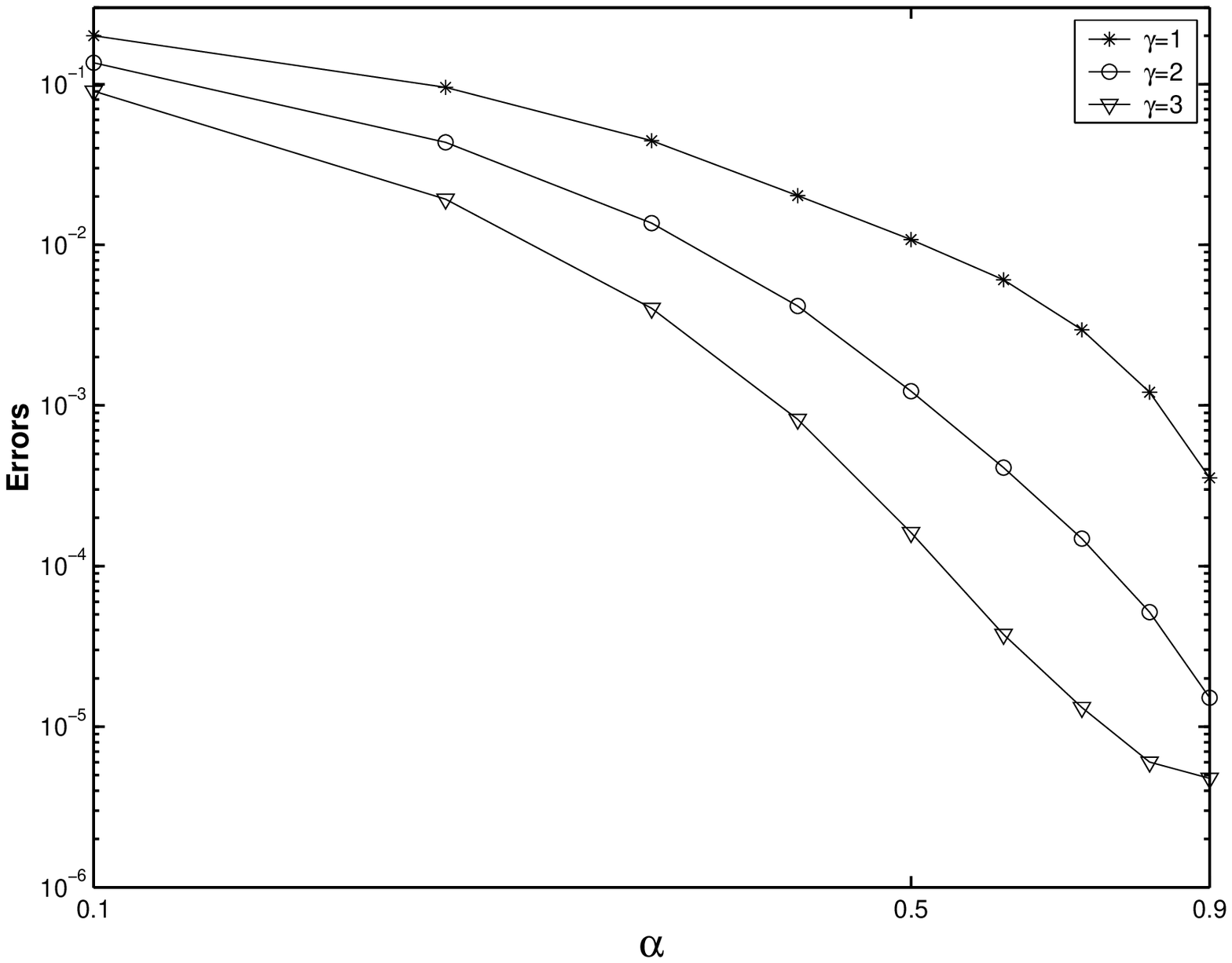}}
 \caption{The errors $|||U_h-u|||_{10}$ for Example 1 plotted against  $\alpha$
for different values of  $\gamma$ and fixed $N=60$ with $m=2$.}\label{fig: error time against alpha}
\end{center}
\end{figure}
%%%%%%%%%%%%%%%%%%%%%%%%%%%%%%%%%%%%%%%%%%%%%%%%%%%%%%%%%%%%%%%%%%5
Next, we test the performance of the spatial finite elements discretizaton (order degree $r$) of the scheme \eqref{eq: DPGFM}. A uniform spatial
mesh that consists of $N_x$ subintervals where each is of width  $h$ will be used. The  time step-size $k$ and the degree of the
 time-stepping DPG discretization are chosen such that
the spatial error is dominating. Hence, from Theorem \ref{thm: main results}, a convergence of order $O(h^{r+1})$ is expected. We illustrated
these results in Table \ref{tab: ||U_h-u|| spatial FEM Example 1}  for $r=1,\,2,\,3.$
 %%%%%%%%%%%%%%%%%%%%%%%%%%%%%%%%%%%%%%%%%%%%%%%%%%%%%%%%%%%%%%%%%%%%
\begin{table}[htb]
\renewcommand{\arraystretch}{1.0}
\begin{center}
\begin{tabular}{|r|rr|rr|rr|}
\hline $N_x$ &\multicolumn{2}{c|}{$r=1$} &\multicolumn{2}{c|}{$r=2$}&\multicolumn{2}{c|}{$r=3$}
\\ \hline
  10& 5.638e-03&      & 1.576e-04&        & 2.633e-06&      \\
  20& 1.426e-03& 1.983& 1.796e-05&  3.133& 1.568e-07& 4.069\\
  30& 6.367e-04& 1.988& 4.730e-06&  3.291& 2.998e-08& 4.081\\
  40& 3.584e-04& 1.998& 1.979e-06&  3.029& 9.172e-09& 4.117\\
  60& 1.592e-04& 2.002& 6.240e-07&  2.847& &\\
 \hline
  \end{tabular}
   \vspace{0.05in} \caption {The errors $|||U_h-u|||_{10}$ for Example 1 with $\alpha=0.5$. We observe a spatial
   convergence of order $h^{r+1}$  for $r=1,\,2,\,3.$ The time mesh will be chosen such that the spatial errors are dominating.}
 \label{tab: ||U_h-u|| spatial FEM Example 1}
\end{center}
\end{table}

\subsection{Example 2} (Less smooth) We choose $u_0$ and $f$ in problem \eqref{eq:fractional equation} such $
 u(x,t) = t^{1-\alpha}{\sin}(\pi x)$ is the exact solution.
It can be seen that the regularity conditions in \eqref{eq:countable-regularity v1}~ hold for $\sigma=1-\alpha$ and $\delta=2-\alpha$. Thus, it
is less smooth than Example 1 in the time variable.  As in the previous example, we demonstrate tabularly and graphically optimal convergence
rates  of the DPG-FE scheme \eqref{eq: DPGFM} in the time direction on the non-uniformly graded meshes in~\eqref{eq: tn standard}. To do so, we
take a relatively large number of subintervals in space and choose $r$ (the degree of the approximate finite element solution in the spatial
variable) appropriately so that the temporal errors are dominating. Thus, by Theorem \ref{thm: main results},   convergence of order
$O(k^{2m+\alpha})$ for  $ \gamma
\ge   \frac{2m+1+\alpha}{1-\alpha} $ is anticipated. However, and as in {\bf Example 1}, the numerical results
in Table \ref{tab: ||U_h-u|| alpha=0.3 Example 2} illustrate more optimistic convergence rates. We observed a uniform global error bounded by
 $Ck^{\min\{\gamma(1-\alpha),m+1\}}$ for $\gamma \ge 1$, which is optimal for
$\gamma\ge (m+1)/(1-\alpha)$ (relaxed) and not for  $ \gamma \ge   \frac{2m+\alpha+1}{1-\alpha} $ as the theory suggested.  For graphical
illustrations, see Figure \ref{fig1: DPG error Example 2}.

%%%%%%%%%%%%%%%%%%%%%%%%%%%%%%%%%%%%%%%%%%%%%%%%%%%%%%%%%%%%%%%%%%%%
\begin{table}[htb]
\renewcommand{\arraystretch}{1}
\begin{center}
\begin{tabular}{|r|rr|rr|rr|rr|rr|}
 \hline
 $N$ &\multicolumn{2}{c|}{$\gamma=1$}
&\multicolumn{2}{c|}{$\gamma=2$}&\multicolumn{2}{c|}{$\gamma=3$}&\multicolumn{2}{c|}{$\gamma=3.6$} &\multicolumn{2}{c|}{$\gamma=4.2$}
\\
\hline \hline
& \multicolumn{10}{c|}{$m=1$}\\
\hline
    20& 9.3e-03&     & 1.5e-03&     & 1.5e-03&     & & & &\\
    40& 5.6e-03& 0.73& 4.5e-04& 1.7& 3.9e-04& 2.0& & & &\\
    80& 3.5e-03& 0.67& 1.7e-04& 1.4& 9.9e-05& 2.0& & & &\\
   160& 2.2e-03& 0.68& 6.5e-05& 1.4& 2.5e-05& 2.0& & & &\\
   320& 1.4e-03& 0.68& 2.5e-05& 1.4& 5.7e-06& 2.1& & & &\\
 \hline
\hline & \multicolumn{10}{c|}{$m=2$}\\
\hline
  10& 6.3e-03&     & 1.4e-03&     & 3.2e-04&     & 2.4e-04&     & 3.7e-04&     \\
  20& 3.8e-03& 0.71& 5.2e-04& 1.5& 6.7e-05& 2.3& 3.3e-05& 2.9& 3.1e-05& 3.6\\
  30& 2.9e-03& 0.69& 2.9e-04& 1.4& 2.8e-05& 2.2& 1.1e-05& 2.6& 7.5e-06& 3.5\\
  40& 2.4e-03& 0.69& 1.9e-04& 1.4& 1.5e-05& 2.2& 5.5e-06& 2.6& 2.9e-06& 3.3\\
  50& 2.0e-03& 0.69& 1.4e-04& 1.4& 9.2e-06& 2.1& 3.1e-06& 2.6& 1.5e-06& 3.1\\
  80& 1.5e-03& 0.69& 7.2e-05& 1.4& 3.4e-06& 2.1& 9.2e-07& 2.6& &\\
\hline
 \end{tabular}
   \vspace{0.05in} \caption {The errors $|||U_h-u|||_{10}$ for
  different time mesh gradings  with $m=1$ and $m=2$ (that is, piecewise linear and piecewise quadratic DPG time stepping
  solution), and  $\alpha=0.3$.  We observe convergence of order $k^{(1-\alpha)\gamma} (=k^{0.7\gamma})$ for $1\le
\gamma \le (m+1)/(1-\alpha)$.} \label{tab: ||U_h-u|| alpha=0.3 Example 2}
\end{center}
\end{table}
 %%%%%%%%%%%%%%%%%%%%%%%%%%%%%%%%%%%%%%%%%%%%%%%%%%%%%%%%%%%%%%%%%%%%%
 \begin{figure}
 \begin{center}
 \scalebox{0.36}{\includegraphics{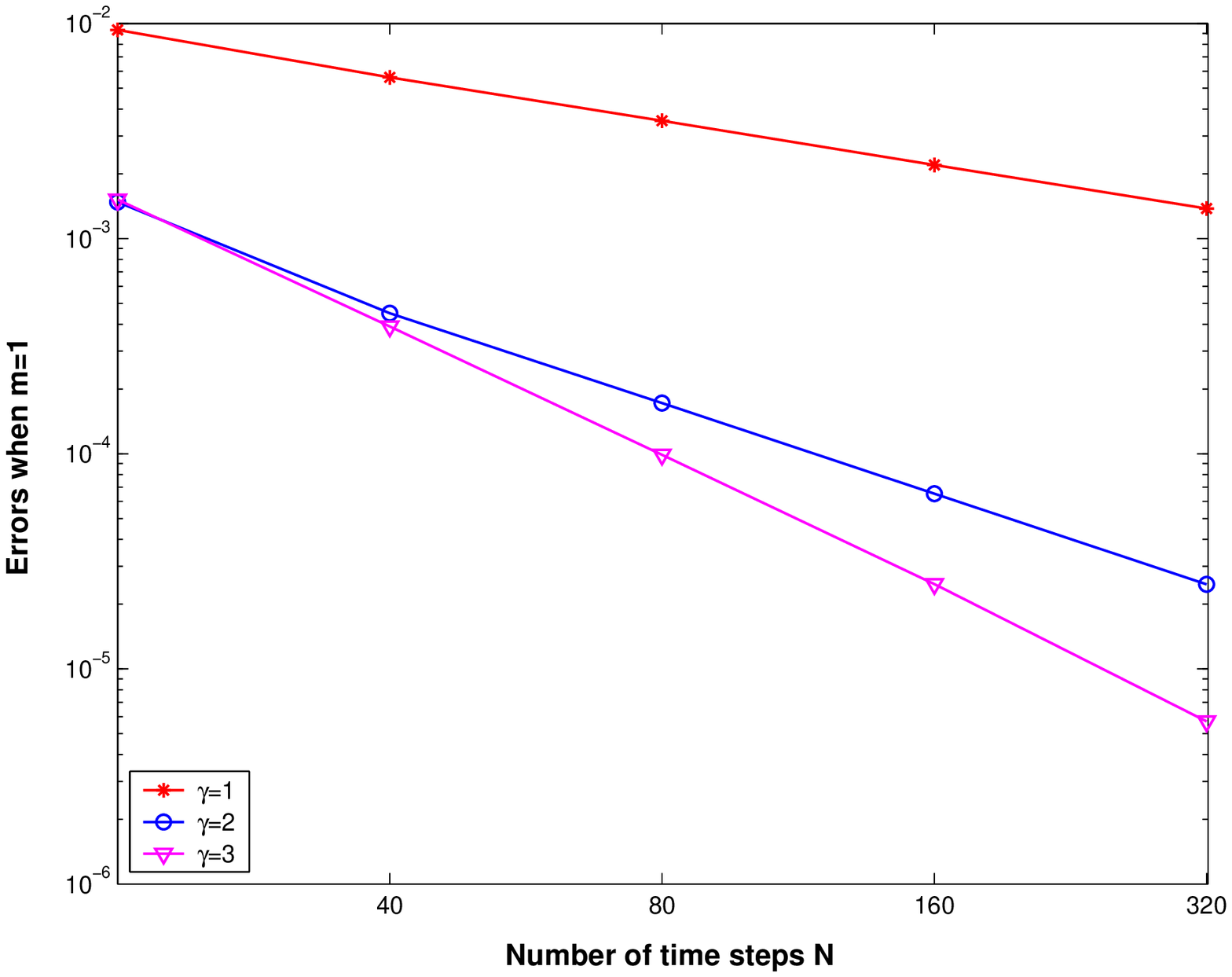} \includegraphics{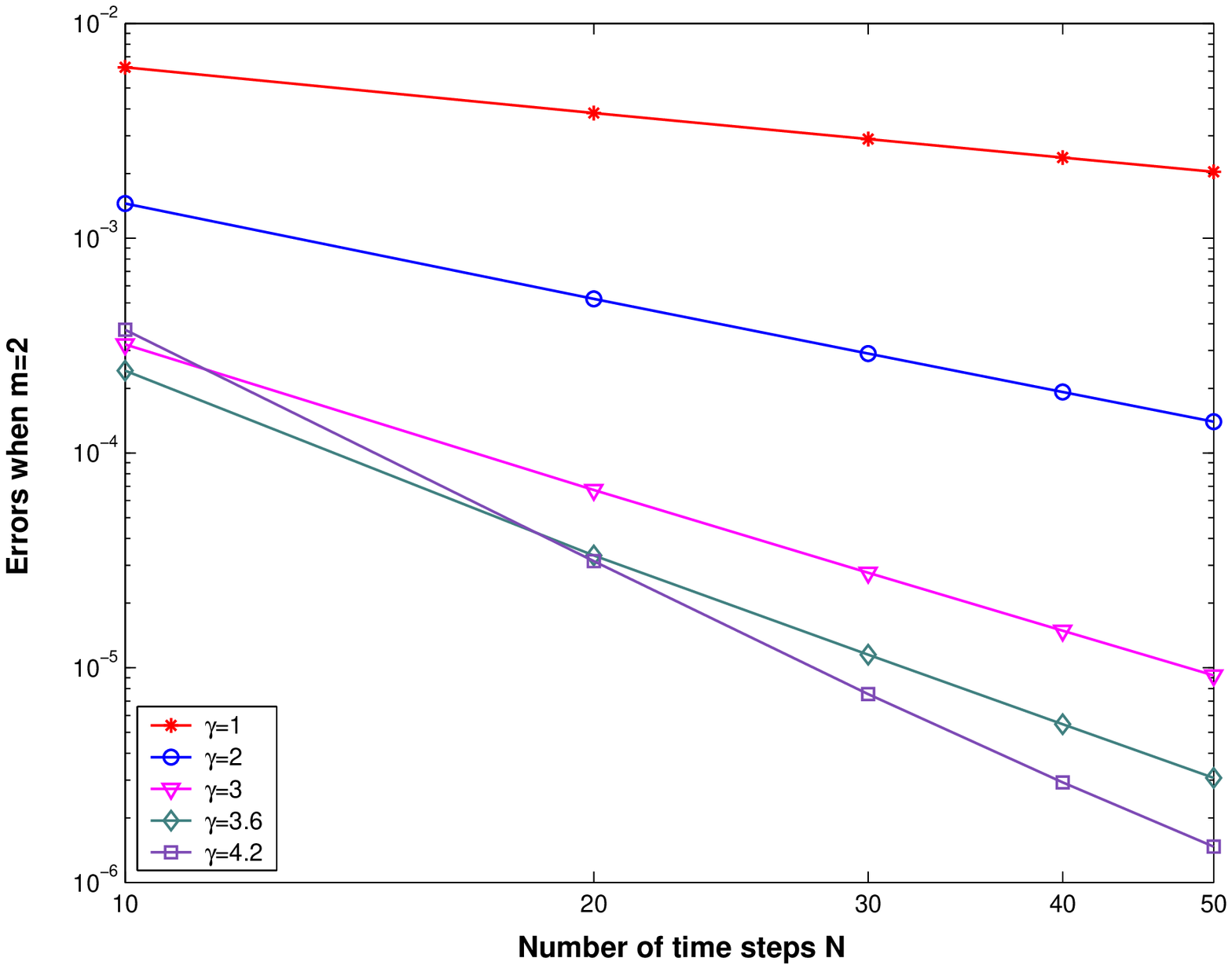}}
  \caption{The errors $|||U_h-u|||_{10}$ for Example 2 plotted against $N$ for different choices of~$\gamma$ and for $m=2,\,3$,
 with $\alpha=0.3$.}\label{fig1: DPG error Example 2}
 \end{center}
 \end{figure}
%%%%%%%%%%%%%%%%%%%%%%%%%%%%%%%%%%%%%%%%%%%%%%%%%%%%%%%%%%%%%%%%%%%%
%%%%%%%%%%%%%%%%%%%%%%%%%%%%%%%%%%%%%%%%%%%%%%%%%%%%%%%%%%%%%%%%%%%%%

\end{document}